\documentclass[journal,twoside,web]{IEEEtran}

\usepackage{amsmath,amsfonts,amssymb} 
 
\usepackage{amsthm}

\usepackage[ruled,vlined,linesnumbered]{algorithm2e}	% Pseudocode/algorithm environments
\usepackage[colorlinks,allcolors=blue]{hyperref}        % Support for hyperlinking
\usepackage{algorithmic}                                % IEEE: Algorithm environments
\usepackage{textcomp}                                   % IEEE: Allow text comparison scripts
\usepackage{enumerate}                                  % More functionalities for the {enumerate} env
\usepackage{graphicx} 						            % Enhanced support for graphics
\usepackage{xcolor} 						            % Enhanced support for graphics
\usepackage{empheq}                                     % (Math): Provides the {empheq} environment                                   
\usepackage{tikz}                                       % TiKz, for the copyright notice

\graphicspath{{figs}}

\newtheorem{theorem}{Theorem}
\newtheorem{lemma}[theorem]{Lemma}
\newtheorem{corollary}{Corollary}[theorem]

\newtheorem{assumption}{Assumption}

\newtheorem{remark}{Remark}

\theoremstyle{definition}
\newtheorem{definition}{Definition}

\newcommand*{\tran}{\mathsf{T}}

\DeclareMathOperator*{\argmin}{arg\,min}
\DeclareMathOperator*{\minimize}{minimize}
\DeclareMathOperator*{\subjectto}{subject~to}
\DeclareMathOperator{\proj}{\mathrm{proj}}

\newcommand\copyrighttext{%
  \footnotesize This is the Accepted Manuscript of an article to appear in the IEEE Transactions on Cybernetics, available at: \url{https://doi.org/10.1109/TCYB.2026.3711940}
}
\newcommand\copyrightnotice{%
    \begin{tikzpicture}[remember picture,overlay]
    \node[anchor=south,yshift=10pt] at (current page.south) {\fbox{\parbox{\dimexpr\textwidth-\fboxsep-\fboxrule\relax}{\copyrighttext}}};
    \end{tikzpicture}%
}

\title{A system-level approach to generalized feedback Nash equilibrium seeking in partially observed games}

\author{
    Otacilio B. L. Neto, Michela Mulas, and Francesco Corona
    \thanks{Otacilio B. L. Neto and Francesco Corona are with the Department of Chemical and Metallurgical Engineering, Aalto University, Finland (e-mails: otacilio.neto@aalto.fi, franciscu.corona@gmail.com). }
    \thanks{Michela Mulas is with the Department of Teleinformatics Engineering, Federal University of Cear\'a, Brazil (e-mail: michela.mulas@ufc.br).}
}

\begin{document}
% -----------------------------------------------------------------------------

\maketitle%
\copyrightnotice%
\begin{abstract} 
    This work proposes an algorithm for seeking generalized feedback Nash equilibria (GFNE) in noncooperative dynamic games.
    The focus is on cyber-physical systems with dynamics which are linear, stochastic, potentially unstable, and partially observed.
    We employ System Level Synthesis (SLS) to reformulate the problem as the search for an equilibrium profile of closed-loop responses to noise, which can then be used to reconstruct a stabilizing output-feedback policy.
    Under this setup, we leverage monotone operator theory to design a GFNE-seeking algorithm capable to enforce closed-loop stability, operational constraints, and communication constraints onto the control policies. 
    This algorithm is amenable to numerical implementation and we provide conditions for its convergence.
    We demonstrate our approach in a simulated experiment on the noncooperative stabilization of a decentralized power grid.
\end{abstract}


% -----------------------------------------------------------------------------
\section{Introduction} \label{sec: Introduction}

\IEEEPARstart{C}{yber-physical} systems are often comprised of many interacting subsystems each operated by a self-interested decision-making agent. 
Each agent should control its subsystem according to a feedback policy which is optimal, given its local objective, while still satisfying global requirements, given the policies applied to the other subsystems. 
However, due to the noncooperative and multi-objective nature of the problem, a traditional approach to policy synthesis becomes impractical.
Dynamic game theory offers an alternative framework that explicitly accounts for the behavior of these rational agents when choosing their policies \cite{Basar1998}.
Under this framework, the task translates to the search for a profile of feedback policies (the \textit{strategies}) which are feasible and agreeable to all agents (the \textit{players}); that is, a strategic equilibrium.
Game-theoretical methods for online decision-making have found success in many applications, such as power systems \cite{Chen2010, MohsenianRad2010, Saad2012}, communication networks \cite{Pavel2012}, and multi-agent robotics \cite{Fisac2019, Spica2020, Wang2021}. 

Due to their generality, a pertinent class of problems concern noncooperative dynamic games in which the underlying system is stochastic, potentially unstable, and partially observed. 
In such problems, players choose output-feedback policies to jointly stabilize the whole system while using (noisy) partial measurements of its internal state.
Their choice can be further restricted by coupled constraints on their control actions, on the state responses, and on the policies themselves (e.g., encoding communication delays).
A relevant solution concept is the \textit{generalized feedback Nash equilibrium} (GFNE, \cite{Basar1998,Facchinei2010}):
A set of policies which are, simultaneously, the best strategy for each player given the others' choices.
Despite prevalent, Nash equilibria are notoriously hard to compute \cite{Daskalakis2009} and there are still no systematic solutions to the aforementioned class of dynamic games.
The current practice for solving dynamic games includes dynamic programming \cite{Kamalapurkar2014, Tanwani2019, Laine2023, Monti2024, Xue2025}, complementarity methods \cite{Reddy2017,Reddy2019}, and \textit{ad-hoc} heuristics \cite{Fridovich-Keil2020a, Fridovich-Keil2020b, Hall2025}. 
These methods are mostly limited to state-feedback problems and cannot address either closed-loop stability, operational constraints, or the design of the policy structure.
Moreover, they focus on solving the game for the sake of analyzing how the agents behave when operating under equilibrium policies; often not discussing how these agents would themselves learn such policies.
\textit{Equilibrium-seeking algorithms}, when players actively search for an equilibrium solution, mimic the behavior of real rational agents and thus provide practical routines for decentralized policy learning.

Recently, monotone operator theory has gained attention as a unifying framework for designing equilibrium-seeking algorithms for many applications \cite{Yi2019, Kaihong2019, Belgioioso2022a, Kaihong2022, Belgioioso2025}.
Under a refined solution concept, the \textit{variational generalized Nash equilibrium} (vGNE), noncooperative games can be reformulated as monotone inclusion problems and then approached by fixed-point methods which often translate to equilibrium-seeking routines.
This operator-theoretical approach not only allows the solution of noncooperative games with non-trivial information structures, a major challenge in the field \cite{Basar2014}, but also simplifies the convergence analysis of such routines. 
Regrettably, its application is mostly limited to static games due to technical challenges arising in GFNE-seeking problems:
The space of output-feedback policies (i.e., mappings from measurement to action signals) is generally not endowed with an inner product, necessary for defining \textit{variational} equilibria.
In addition, the numerics common to most vGNE-seeking routines (e.g., computing pseudo-gradient and projections) are cumbersome when the strategy spaces are infinite-dimensional; thus limiting their scalability.
However, if such issues can be overcome, monotone operator theory stands out as a promising platform for designing GFNE-seeking algorithms for dynamic games.
This is the purpose of this study.

In this work, we propose a GFNE-seeking algorithm for partially observed stochastic dynamic games.
Our algorithm is centred on the equivalent representation of linear output-feedback policies as their corresponding \textit{closed-loop responses} (or \textit{system level responses}) to disturbances.
Under this setup, players design policies indirectly by instead choosing desired system level responses that serve as their parametrization.
Using the System Level Synthesis (SLS, \cite{Anderson2019}) methodology, we translate the original dynamic game into a static game whose strategy space consists of closed-loop responses that explicitly enforce stability (for the whole networked system), operational constraints (on the state and control signals), and communication constraints (for the output-feedback policies).
This strategy space is convex and finite-dimensional, making the associated static game amenable to numerical solutions.
We formulate a monotone inclusion problem to obtain a vGNE for this static game, which is then used to directly reconstruct a GFNE solution to the original problem; effectively enabling the design of a GFNE-seeking algorithm. 
In this direction, we design a routine in which players improve their policies, in parallel, by updating their system level parametrization, while a coordinator ensures that constraints are satisfied. 
Using results from operator theory, we derive convergence certificates for this algorithm.
In summary, our contributions are:
\begin{enumerate}[(i)]
    \item 
    A principled approach to solve noncooperative dynamic games using monotone operator theory. 
    We propose an equilibrium-seeking algorithm based on the equivalent representation of GFNE as their corresponding closed-loop responses to external disturbances;
    \item 
    A configuration of our method to address stability, operational, and informational constraints in partially observed systems. 
    We focus on the (not restrictive) case of policies subjected to actuation and communication delays;
    \item 
    A formal analysis of the convergence properties of our algorithm and a specific implementation of its numerics to improve its efficiency and scalability.
\end{enumerate}
We demonstrate our GFNE-seeking algorithm in a simulated problem: The noncooperative control of a decentralized power grid. 
The routine is executed alongside the operation of the system and we show that players approach a GFNE while still complying with operational and communication constraints.

This article builds on our previous work \cite{NMC_SLS_BRD2024} which addressed GFNE-seeking in state-feedback problems (i.e., fully-observed games) using the \textit{best-response dynamics} algorithm. 
Our results here provide a significant extension not only by considering the more general class of output-feedback problems, but also by applying monotone operator theory to design a scalable \textit{operator splitting} algorithm with well-understood convergence properties. 
Moreover, while our previous study focused on systems affected by bounded disturbances, here we consider the case of unbounded (e.g., Gaussian) noise. 

The paper is structured as follows: 
Section \ref{sec: Preliminaries} reviews noncooperative games and operator splitting algorithms. 
Section \ref{sec: GFNE_Seeking} presents the SLS-based reformulation of GFNE problems and our proposed equilibrium-seeking algorithm for their solution.
Section \ref{sec: Examples} demonstrates our proposal in an exemplary problem and, finally, Section \ref{sec: Concluding_Remarks} concludes the article. 

% -----------------------------------------------------------------------------
\subsection{Notation} \label{sec: Introduction_Notation}

We use Latin and Greek letters to denote vectors (lowercase) and mappings (uppercase). 
Sets are in calligraphic font; exceptions are the usual $\mathbb{R}$ and $\mathbb{N}$. 
Given $\mathcal{I} \subseteq \mathbb{N}$, sequences are denoted $x = (x_t)_{t\in\mathcal{I}}$ or $x = (x_t)_{t=0}^{T}$ if $\mathcal{I} = \{0,\ldots,T\}$. 
For $p \in (0,\infty)$, we define the set of $N_x$-dimensional sequences $\ell^{N_x}_p(\mathcal{I}) = \{ x \mid \| x \|_{\ell_p} = (\sum_{t\in\mathcal{I}} \| x_t \|^p)^{\frac{1}{p}} < \infty \}$, with $\ell^{N_x}_{\infty}$ the space of all bounded sequences. 
$\mathcal{L}(\mathcal{X},\mathcal{Y})$ is the set of bounded linear operators $A : \mathcal{X} \to \mathcal{Y}$. 
We use the standard definitions of Hardy spaces $\mathcal{H}_{\infty}$ and $\mathcal{RH}_{\infty}$, with $\frac{1}{z}\mathcal{RH}_{\infty}$ denoting the set of real-rational strictly proper transfer functions. 
Some standard signals and operators used in this paper include: 
The impulse signal $\delta = (\delta_{t})_{t\in\mathbb{N}}$, the identity operator $I$ and matrix $I_{N}$ (with $e_i$ denoting its $i$-th column), and the normal cone $N_{\mathcal{S}}$ and projection $\proj_{\mathcal{S}}$ operators for a closed convex set $\mathcal{S} \subseteq \mathbb{R}^{N}$.

We distinguish set-valued mappings from ordinary functions using the notation $F : \mathcal{X} \rightrightarrows \mathcal{Y}$. 
For any tuple $s = (s^p)_{p\in\mathcal{P}} \in \mathcal{S}$ we often write $s = (s^p, s^{-p})$ to highlight the $p$-th element; this should not be interpreted as a reordering. 
Similarly, for any set $\mathcal{S} = \prod_{p\in\mathcal{P}}\mathcal{S}^{p}$, we define the product $\mathcal{S}^{-p} = \prod_{\tilde{p}\in\mathcal{P}\backslash\{p\}}\mathcal{S}^{\tilde{p}}$.
A mapping $F : \mathbb{R}^{N} \rightrightarrows \mathbb{R}^M$ is $M_F$-strongly-monotone and $L_F$-Lipschitz continuous with $0 < M_F \leq L_F < \infty$ if 
\begin{equation*}
    \langle u - v, x - y \rangle \geq M_F \| x - y \|_2^2  \text{ and } \|u - v\|_2 \leq L_F \| x - y \|_2,
\end{equation*}
for any $x,y \in \mathbb{R}^N$, $u \in F(x)$ and $v \in F(y)$.
$F$ is a contraction if $L_F < 1$, nonexpansive if $L_F = 1$, monotone if the first inequality only holds for $M_F = 0$, and maximally monotone if no other monotone operator properly contains it.
Finally, $\mathbf{fix}(F) = \{ x \mid x \in F(x) \}$ and $\mathbf{zer}(F) = \{ x \mid 0 \in F(x) \}$ are, respectively, the set of fixed-points and zeroes of $F$. 

% -----------------------------------------------------------------------------
\section{Preliminaries: Noncooperative games and equilibrium seeking via operator splitting} \label{sec: Preliminaries}

A (static) $N_P$-player game, denoted by a tuple 
\begin{equation} \label{eq: Static_Game}
    \mathcal{G} \coloneqq (\mathcal{P},\{ S^p \}_{p\in\mathcal{P}},\{ L^p \}_{p\in\mathcal{P}}),
\end{equation}
defines the problem in which \textit{players} $p \in \mathcal{P} = \{1, \ldots, N_P\}$ each decides on a strategy $s^p \in S^p(s^{-p}) \subseteq \mathcal{S}^p$ to minimize an objective function $L^p : \mathcal{S}^1 \times \cdots \times \mathcal{S}^{N_P} \to \mathbb{R}$. 
The strategy sets $\mathcal{S}^p \subseteq \mathbb{R}^{N_s^p}$ ($\forall p\in\mathcal{P}$) define the actions available to the players, with $S^{p} : \mathcal{S}^{-p} \rightrightarrows \mathcal{S}^p$ restricting this choice based on their rivals' strategies. 
As such, the problem is coupled through both objective functions and feasible action sets. 
Finally, the players are assumed to be non-cooperative and acting simultaneously. 

A strategy profile $s = (s^1, \ldots, s^{N_P}) \in \mathcal{S} = \mathcal{S}^1 \times \cdots \times \mathcal{S}^{N_P}$, characterizes a solution to $\mathcal{G}$ when it is agreeable to all players. 
If acting rationally, a reasonable assumption is that players might prefer to (myopically) choose a best-response to their rivals's strategies. 
Formally, the mapping $BR^p : \mathcal{S}^{-p} \rightrightarrows \mathcal{S}^{p}$,
\begin{equation} \label{eq: Best_Response}
    BR^p(s^{-p}) \coloneqq \textstyle\argmin_{s^p} \{ L^p(s^p,s^{-p}) \mid s^p \in S^p(s^{-p}) \}
\end{equation}
defines all such strategies for each player $p \in \mathcal{P}$. 
A profile is a best-response for all players, simultaneously, when it is a fixed-point of the joint-best-response mapping $BR : \mathcal{S} \rightrightarrows \mathcal{S}$ defined as $BR(s) \coloneqq BR^1(s^{-1}) \times \cdots \times BR^{N_P}(s^{-N_P})$. 
This motivates a formal solution concept for non-cooperative games known as the \textit{generalized Nash equilibrium} (GNE, \cite{Facchinei2010}). 
\begin{definition} \label{def: Nash_Equilbrium}
    A profile $s^{\star} = (s^{1^{\star}},\ldots,s^{N_P^{\star}}) \in \mathcal{S}$ satisfying $s^{\star} \in BR(s^{\star})$ is a \textit{generalized Nash equilibrium} (GNE) of $\mathcal{G}$.
\end{definition}
Under this definition, the game $\mathcal{G}$ is solved when no player can improve its outcome by unilaterally deviating from its current strategy. 
The set of fixed-points $\mathbf{fix}(BR)$ comprise all the solutions to $\mathcal{G}$. 
As such, the existence and uniqueness of a GNE depend explicitly on the properties of the objectives $\{ L^p \}_{p\in\mathcal{P}}$ and constraints $\{ S^p \}_{p\in\mathcal{P}}$ functions. 
Hereafter, we ensure that the problems being discussed are well-posed by making the following assumptions on their primitives:
\begin{assumption} \label{as: Nash_Equilbrium_Existence}
    For each player $p \in \mathcal{P}$,
    \begin{enumerate}[(a)]
        \item 
        the objective $L^p$ is continuously differentiable on $\mathcal{S}$, and $L^p( \cdot, s^{-p} )$ is strictly convex on $\mathcal{S}^p$ for all $s^{-p} \in \mathcal{S}^{-p}$.
        \item
        the strategy set $\mathcal{S}^{p}$ is nonempty, compact, and convex. 
        The mapping $S^p : \mathcal{S}^{-p} \rightrightarrows \mathcal{S}^{p}$ takes the form
        \begin{equation} \label{eq: Definition_Sp}
            S^p(s^{-p}) \coloneqq \{ s^p \in \mathcal{S}^p \mid (s^p, s^{-p}) \in \mathcal{S}^{\text{global}} \},
        \end{equation}
        given a compact convex constraint set $\mathcal{S}^{\text{global}} \subseteq \mathbb{R}^{N_s}$.
        Moreover, they satisfy $\mathcal{S}_{\mathcal{G}} = (\mathcal{S}^1 \times \cdots \times \mathcal{S}^{N_P}) \cap \mathcal{S}^{\text{global}} \neq \emptyset$.
    \end{enumerate}
\end{assumption}
Under Assumption \ref{as: Nash_Equilbrium_Existence}, $BR : \mathcal{S}_{\mathcal{G}} \to \mathcal{S}_{\mathcal{G}}$ is a continuous function from a nonempty compact convex set onto itself: 
The existence of GNE (i.e., $\mathbf{fix}(BR) \ne \emptyset$) is ensured by the Brouwer's fixed-point theorem \cite{Conway2007}. 
In practice, these conditions imply that players' best-responses are always unique and their actions are coupled only through a common constraint. 
Specifically, Eq.~\eqref{eq: Definition_Sp} can be interpreted as the competition for a limited shared resource (e.g., bandwidth in a network or goods in a supply-chain). 
Although restrictive, these assumptions still cover a broad class of problems of practical relevance. 

We investigate GNE-seeking algorithms for solving $\mathcal{G}$. 
Namely, we are interested in fixed-point methods for players to learn a strategy profile $s^{\star} \in \mathbf{fix}(BR)$. 
Assuming that $\mathcal{G}$ can be repeated, with $s_k = (s_k^1, \ldots, s_k^{N_P})$ being the strategies played at the $k$-th episode, we seek an \textit{update rule} $T : \mathcal{S}_{\mathcal{G}} \to \mathcal{S}_{\mathcal{G}}$ such that $s_{k+1} = T(s_k)$ converges to an equilibrium $s^{\star} \in \mathbf{fix}(BR)$. 
Conforming to realistic scenarios, such a mapping must be
\begin{itemize}
    \item \textit{semi-decentralized}, in the sense that players compute their updates independently, with minimal coordination;
    \item \textit{based on private information}, in the sense that players do not query their rivals' objectives and strategy sets.
\end{itemize}
A natural choice consists on the operator $T = (1{-}\eta)I + \eta BR$ for some $\eta \in (0,1)$, since $s_{k+1} = (1{-}\eta)s_k + \eta BR(s_k)$ converges to a fixed-point $s^{\star} \in \mathbf{fix}(BR)$ (a GNE, by definition) when $BR$ is either a nonexpansive or contractive operator \cite{Bauschke2017}. 
This method, \textit{Best-Response Dynamics} (BRD), satisfies the above requirements. 
However, the Lipschitz properties of $BR$ are difficult to establish for most generalized Nash equilibrium problems and convergence might fail even in simple cases \cite{NMC_SLS_BRD2024}. 

A class of equilibrium-seeking algorithms which are computationally tractable and have straightforward convergence properties can be obtained by adopting a refined solution concept: 
The \textit{variational generalized Nash equilibrium} (vGNE, \cite{Facchinei2010}). 
\begin{definition} \label{def: Variational_Nash_Equilbrium}
    A profile $s^{\star} = (s^{1^{\star}},\ldots,s^{N_P^{\star}}) \in \mathcal{S}$ satisfying 
    \begin{equation} \label{eq: Variational_Nash_Equilibrium}
        \langle F(s^{\star}), s - s^{\star} \rangle \geq 0, \quad \forall s \in \mathcal{S}_{\mathcal{G}},
    \end{equation}
    given the {pseudo-gradient} $F(s) = (\nabla_{s^p} L^p(s^p, s^{-p}))_{p\in\mathcal{P}}$, is a \textit{variational generalized Nash equilibrium} (vGNE) of game $\mathcal{G}$. 
    Moreover, we let $\mathbf{VI}(BR)$ denote the set of all vGNE of $\mathcal{G}$.
\end{definition}
The term \textit{vGNE} is justified since $\mathbf{VI}(BR) \subseteq \mathbf{fix}(BR)$ \cite{Facchinei2010}. 
Despite the implication that some GNE can potentially be lost, the variational approach always leads to equilibrium solutions which are fair, in the sense that the ``shared resource'' described by $\mathcal{S}^{\text{global}}$ is uniformly distributed among players; 
it is indeed a refinement \cite{Kulkarni2012}. 
The following standing assumption is taken: 
\begin{assumption} \label{as: F_MaximalMonotonicity}
    The pseudo-gradient $F$ is maximal monotone.
\end{assumption}
Under this assumption, obtaining a vGNE $s^{\star} \in \mathbf{VI}(BR)$ is equivalent to the monotone inclusion problem 
\begin{equation} \label{eq: vGNE_Monotone_Inclusion_Problem}
    \text{Find } s^{\star} \in \mathbb{R}^{N_s} \text{ such that } 0 \in F(s^{\star}) + N_{\mathcal{S}_{\mathcal{G}}}(s^{\star}),
\end{equation}
which can be solved via operator splitting methods \cite{Bauschke2017,Ryu2022}. 
In this work, we tackle Problem \eqref{eq: vGNE_Monotone_Inclusion_Problem} using \textit{forward-backward splitting} (FBS), a well-established method for solving monotone inclusion problems. 
This method and its application to vGNE-seeking are overviewed in the following.

% -----------------------------------------------------------------------------
\subsection{vGNE-seeking via FBS} \label{sec: Operator_Splitting}

Because $F$ is maximal monotone and single-valued, some algebra allows us to obtain the equivalence
\begin{equation*}
    0 \in F(s^{\star}) + N_{\mathcal{S}_{\mathcal{G}}}(s^{\star})  \iff s^{\star} = (I + \eta N_{\mathcal{S}_{\mathcal{G}}})^{-1}(I - \eta F)(s^{\star}),
\end{equation*}
for any constant $\eta > 0$. In turn, this implies
\begin{equation}
    \mathbf{VI}(BR) = \mathbf{zer}(F + N_{\mathcal{S}_{\mathcal{G}}}) = \mathbf{fix}(T)
\end{equation}
for the operator $T = (I + \eta N_{\mathcal{S}_{\mathcal{G}}})^{-1} \circ (I - \eta F)$. 
This operator is a candidate update rule since the iterations $s_{k+1} = T(s_k)$ converge to a (unique) solution $s^{\star} \in \mathbf{VI}(BR)$ whenever $T$ is a contraction. 
Due to its composition form, this iteration can be split into a forward and backward step, respectively 
\begin{subequations}
\begin{align}
    \underbrace{\begin{bmatrix} s^{1}_{k+1/2} \\ \vdots \\ s^{N_P}_{k+1/2} \end{bmatrix}}_{s_{k+1/2}}
        &= \underbrace{\begin{bmatrix} s^1_k - \eta \nabla_{s^1} L^1(s^1_k, s^{-1}_k) \\ \vdots \\ s_k^{N_P} - \eta \nabla_{s^{N_P}} L^p(s^{N_P}_k, s^{-N_P}_k) \end{bmatrix}}_{(I - \eta F)(s_k)};   \label{eq: FB_Equations_a} \\ 
    s_{k+1} &= \proj_{\mathcal{S}_{\mathcal{G}}}(s_{k+1/2}),  \label{eq: FB_Equations_b}
\end{align} \label{eq: FB_Equations}%
\end{subequations}
since $(I + \eta N_{\mathcal{C}})^{-1} = \proj_{\mathcal{C}}$ for any closed convex set $\mathcal{C}$ \cite{Bauschke2017}. 
Finally, we obtain the semi-decentralized routine for vGNE-Seeking via FBS outlined in Algorithm \ref{alg: FB_vGNE-Seeking}. 
At each $k > 0$, this routine consists of two main instructions:
\begin{enumerate}[(1)]
    \item Simultaneously, players propose new strategies to improve their objectives while disregarding the constraints;
    \item A coordinator collects these proposals and then computes (and broadcasts) the closest permissible strategies.
\end{enumerate}

\begin{algorithm}[htb!] \label{alg: FB_vGNE-Seeking}
    \caption{vGNE-Seeking via FBS}
    Initialize $s_{0} = (s_{0}^1,\ldots,s_{0}^{N_P})$\;
    \For{$k = 0, 1, 2, \ldots$}{ 
        \For{each player $p \in \mathcal{P}$}{
            $s^p_{k+1/2} \coloneqq s^p_{k} - \eta \nabla_{s^p} L^p(s^p_k, s^{-p}_k)$\;
        }
        \For{coordinator}{
            $s_{k+1} \coloneqq \proj_{\mathcal{S}_{\mathcal{G}}}(s_{k+1/2})$\;\vspace*{-0.2em}
        }\vspace*{-0.2em}
    }
\end{algorithm}

In practice, this routine is interpreted as players moving towards best-response strategies while a coordinator ensures that no constraint is violated. 
In addition to this convenient interpretation, Algorithm \ref{alg: FB_vGNE-Seeking} benefits a realistic vGNE-seeking by requiring only the profiles $(s_k)_{k\in\mathbb{N}}$ to be public: 
The players do not access the objectives $\{ L^p \}_{p\in\mathcal{P}}$ and constraint $\{ S^p \}_{p\in\mathcal{P}}$ functions, except for their own. 
This information needs only to be available to the coordinator. 
We remark that the routine can still be specialized to have the players estimate their rivals' strategies by exchanging information \cite{Yi2019, Belgioioso2022a, Kaihong2019, Kaihong2022}.

As previously mentioned, $T(s_k) \to s^{\star} \in \mathbf{VI}(BR)$ when $T$ is a contraction. 
Because the operator $\proj_{\mathcal{S}_{\mathcal{G}}} = (I + \eta N_{\mathcal{S}_{\mathcal{G}}})^{-1}$ is nonexpansive (as $\mathcal{S}_{\mathcal{G}}$ is nonempty, closed and convex), this translates into requiring that the forward operator $(I - \eta F)$ be contractive. 
The following result can thus be established:
\begin{lemma} \label{thm: fFB_Convergence}
    Let the pseudo-gradient $F$ be $L_F$-Lipschitz and $M_F$-strongly-monotone.
    If $\eta \in (0, 2M_F/L_F^2)$, then Algorithm \ref{alg: FB_vGNE-Seeking} converges linearly to the unique $s^{\star} \in \mathbf{VI}(BR)$ with rate
    \begin{equation} 
        \lim_{k \to \infty} \frac{\| s_{k+1} - s^{\star} \|_2}{\| s_k - s^{\star} \|_2} = \sqrt{1 - \eta (2M_F - \eta L_F^2)}
    \end{equation}
    from any initial strategy profile $s_0 \in \mathbb{R}^{N_s}$. 
\end{lemma}
\begin{proof}
    Since $\mathbf{VI}(BR) = \mathbf{zer}(F + N_{\mathcal{S}_{\mathcal{G}}})$, the proof is as in \cite[Proposition 26.16]{Bauschke2017} with $A = N_{\mathcal{S}_{\mathcal{G}}}$ and $B = F$.
\end{proof}
A corollary to this result is that Problem \eqref{eq: vGNE_Monotone_Inclusion_Problem} has a unique solution when $\kappa_F = L_F/M_F \in (0,\infty)$. 
This constant is the \emph{condition number} of $F$, as it describes the sensitivity of this operator to changes in its arguments. 
The linear convergence of Algorithm \ref{alg: FB_vGNE-Seeking}, implying that each strategy update decreases the distance to the vGNE solution by (at least) a constant factor, can still be slow if $\kappa_F \gg 0$.
In such cases, the game (or, more directly, the pseudo-gradient $F$) is said to be ill-conditioned. 
We also note that choosing a learning rate $\eta \in (0, 2M_F/L_F^2)$ requires full knowledge of the pseudo-gradient.
In this work, it is assumed that the coordinator is responsible for computing and broadcasting this parameter to the players. 

% -----------------------------------------------------------------------------
\subsection{Noncooperative dynamic games} \label{sec: Dynamic_Games}

We are interested in $N_P$-player dynamic stochastic games, 
\begin{equation} \label{eq: Dynamic_Game}
    \mathcal{G}_{\infty}^{\text{\text{LQ}}} \coloneqq (\mathcal{P}, \{ U^p \}_{p\in\mathcal{P}}, \{ J^p \}_{p\in\mathcal{P}}),
\end{equation}
defined by linear dynamics and measurement model 
\begin{subequations}
    \begin{empheq}[left=\Sigma_{\mathcal{G}}:\empheqlbrace]{align} 
        x_{t+1} &= A x_{t} + B_w w_t + \textstyle\sum_{p\in\mathcal{P}} B_u^p u^p_{t}; \label{eq: DynamicGame_SS_a}\\
          y_{t} &= C x_{t} + D_w w_t.         \label{eq: DynamicGame_SS_b}
    \end{empheq} \label{eq: DynamicGame_SS}%
\end{subequations}
with initial state $x_0 = 0$. 
The dynamics (Eq.~\ref{eq: DynamicGame_SS_a}) describe how the state of the game, $x = (x_t)_{t\in\mathbb{N}}$, evolves in response to the disturbances $w = (w_t)_{t\in\mathbb{N}}$ and player's actions $u^p = (u_t^p)_{t\in\mathbb{N}}$, $p \in \mathcal{P}$. 
The measurement process (Eq.~\ref{eq: DynamicGame_SS_b}) describes how a public observation signal, $y = (y_t)_{t \in \mathbb{N}}$, is formed by (noisy) partial emissions of the state. 
We consider bounded signals, that is, $x \in \ell_{\infty}^{N_x}(\mathbb{N})$, $u^p \in \ell_{\infty}^{N_u^p}(\mathbb{N})$, $w \in \ell_{\infty}^{N_w}(\mathbb{N})$, and $y \in \ell_{\infty}^{N_y}(\mathbb{N})$. 
Finally, we assume $w$ to be a white noise process satisfying $\mathrm{E}w_t = 0$ and $\mathrm{E}(w_{t+\tau} w_t^{\tran}) = \delta_{\tau} I_{N_w}$ for all $t,\tau \in \mathbb{N}$.

In this class of dynamic games, each player decides a plan of action $u^p \in U^p(u^{-p})$ to minimize an objective functional
\begin{equation} \label{eq: DynamicGame_Cost}
    J^p(u^p, u^{-p}) \coloneqq 
        \mathrm{E}\left[\sum_{t=0}^{\infty} \left\| \begin{bmatrix} W^p_x & W^p_u \end{bmatrix}\hspace*{-0.25em}\begin{bmatrix} x_t \\ u^p_t \end{bmatrix} \right\|_2^2 \right],
\end{equation}
given weighting matrices $W^p_x \in \mathbb{R}^{N_z \times N_x}$ and $W^p_u \in \mathbb{R}^{N_z \times N_u}$ with dimension $N_z \geq N_x + N_u$. 
This objective corresponds to a standard \emph{linear-quadratic regulator} control objective and can be interpreted as the expected energy $\mathrm{E}\| z \|_{\ell_2}^2$ of a performance signal $z = (W_x^p x_t + W_u^p u_t^p)_{t\in\mathbb{N}}$~\cite{Zhou1998}. 
Note that $J^p(\cdot)$ depends on other players' actions $u^{-p}$ via the state signal $x$ (Eq.~\ref{eq: DynamicGame_SS_a}).
The mapping $U^p : \mathcal{U}^{-p} \rightrightarrows \mathcal{U}^p$ restricts the $p$-th player's strategies based on its rivals' choices, with $\mathcal{U}^p$ being the set of all permissible strategies. 
As in the static case, players might prefer to act according to best-responses,
\begin{equation*} \label{eq: Best_Response_GNE_Dynamic_Map}
    BR^p(u^{-p}) \coloneqq \textstyle\argmin_{u^p} \big\{ J^p(u^p, u^{-p}) \mid u^p \in U^p(u^{-p}) \big\},
\end{equation*}
so that a (open-loop) GNE solution to $\mathcal{G}_{\infty}^{\text{\text{LQ}}}$ is understood as an action profile $u^{\star} = (u^{1^{\star}}, \ldots, u^{N_P^{\star}}) \in \mathcal{U} = \mathcal{U}^1 \times \cdots \times \mathcal{U}^{N_P}$, for which no player unilaterally benefits by deviating. 
Formally, this corresponds to a fixed-point $u^{\star} \in \mathbf{fix}(BR)$. 
The following assumptions are thus taken for this class of games:
\begin{assumption} \label{as: FeedbackNash_Equilbrium_Existence}
    For each player $p \in \mathcal{P}$,
    \begin{enumerate}[(a)]
        \item 
        the weighting matrix $W_u^p$ is full-column-rank and it satisfies $(W_u^p)^{\tran}(W_x^p) = 0$ and $(W_x^p)^{\tran}(W_u^p) = 0$;
        \item 
        the noise-filtering matrices $B_w$ and $D_w$ are full-row-rank and they satisfy $(B_w)(D_w)^{\tran} = 0$ and $(D_w)(B_w)^{\tran} = 0$.
        \item 
        the mapping $U^p : \mathcal{U}^{-p} \rightrightarrows \mathcal{U}^p$ takes the form
        \begin{equation*}
            U^p(u^{-p}) \coloneqq \{ u^p \in \mathcal{U}^{p} \mid (u^p, u^{-p}) \in \mathcal{U}^{\text{global}}  \},
        \end{equation*}
        given the local $\mathcal{U}^p$ and global $\mathcal{U}^{\text{global}}$ constraint sets
        \begin{align*}
            \mathcal{U}^{p}             &= \{ u^p \in \ell_{\infty}^{N_u^p}(\mathbb{N}) \mid G_{u^p} u^p_t     \preceq 1,~ t \in \mathbb{N} \}; \\
            \mathcal{U}^{\text{global}} &= \{ u   \in \ell_{\infty}^{N_u  }(\mathbb{N}) \mid G_x x_t + G_u u_t \preceq 1,~ t \in \mathbb{N} \},
        \end{align*}
        where we implicitly use the fact that $x = F_u u + F_w w$ for causal linear operators $(F_u, F_w)$ induced by Eq.~\eqref{eq: DynamicGame_SS_a}. 
        Moreover,  $\mathcal{U}_{\mathcal{G}} = (\mathcal{U}^1 \times \cdots \times \mathcal{U}^{N_P}) \cap \mathcal{U}^{\text{global}} \neq \emptyset$.
    \end{enumerate}
\end{assumption} 
These conditions are analogous to Assumption \ref{as: Nash_Equilbrium_Existence}: 
They are to ensure that $\mathbf{fix}(BR) \neq \emptyset$. 
In this case, however, the joint policy $u = (u^1, \ldots, u^{N_P})$ must also be internally stabilizing to avoid $x_t \to \infty$ (and thus, $J^p(u^p,u^{-p}) \nless \infty$ for some $p \in \mathcal{P}$). 
Hereafter, we characterize an equilibrium as admissible only if it stabilizes the game. 
The following assumption ensures that (not necessarily $\mathcal{U}_{\mathcal{G}}$-feasible) stabilizing action profiles exist.
\begin{assumption} \label{as: Game_Is_Stabilisable}
    The system $(A, B_u, C)$, given the input matrix $B_u = [B_u^1 ~\cdots~ B_u^{N_P}]$, is both controllable and observable. 
\end{assumption}

In this work, we are not interested in open-loop equilibria $u^{\star} \in \mathbf{fix}(BR)$. 
A plan of action with such an information structure is undesirable, as the game becomes sensible to noise disturbances and decision errors \cite{Basar1998}. 
Conversely, feedback policies $u^p = K^p(y)$, for some $K^p : \ell_{\infty}^{N_y} \to \ell_{\infty}^{N_u^p}$, can detect such errors and adapt the plan of action accordingly. 
Thus, we let each $p$-th player's actions be represented by the policies 
\begin{equation} \label{eq: Policy_CLPS}
    u^p \coloneqq K^p y, \quad K^p : y \mapsto \Phi_K^p * y,
\end{equation}%
given linear causal operators $K^{p} \in S^p(K^{-p}) \subseteq \mathcal{L}(\ell_{\infty}^{N_y},\ell_{\infty}^{N_u^p})$ with convolution kernels $\Phi_K^p = (\Phi^p_{K,n})_{n\in\mathbb{N}}$. 
The mapping $S^p$ restrict the strategies to policies that are stabilizing and satisfy some $\mathcal{G}_{\infty}^{\text{\text{LQ}}}$-related restrictions (e.g., actuation and communication delays). 
In this setup, the $p$-th player's best-responses are the solutions to the synthesis problem 
\begin{subequations}
    \begin{align}
         \minimize_{u^p \coloneq K^p y} \quad 
            & \mathrm{E}\left[\sum_{t=0}^{\infty} \left\| \begin{bmatrix} W^p_x & W^p_u \end{bmatrix}\hspace*{-0.25em}\begin{bmatrix} x_t \\ u^p_t \end{bmatrix} \right\|_2^2 \right]  \label{eq: Best_Response_GFNE_a}\\
        \subjectto_{t = 0,1,\ldots} \quad
            & \Sigma_{\mathcal{G}} \text{ (Eq.~\ref{eq: DynamicGame_SS}) with } u^{\tilde{p}} = K^{\tilde{p}}y ~ (\forall \tilde{p}\in\mathcal{P}),                                 \label{eq: Best_Response_GFNE_b}\\[-1.5ex]    
            & K^p \in S^p(K^{-p}),                                                                                                                                                  \label{eq: Best_Response_GFNE_c}\\   
            & \begin{bmatrix} G_x & G_u \\ & e_p^{\tran} \otimes G_{u^p} \end{bmatrix} \begin{bmatrix} x_t \\ u_t \end{bmatrix} \preceq 1.                                          \label{eq: Best_Response_GFNE_d}
    \end{align} \label{eq: Best_Response_GFNE}%
\end{subequations}%
We denote the solutions to Problem \eqref{eq: Best_Response_GFNE} as $BR_{K}^p(K^{-p})$, given the joint policy $K^{-p} : y \mapsto (K^{\bar{p}}y)_{\bar{p}\in\mathcal{P}\backslash\{p\}}$. 
As before, the map 
\begin{equation*}
    BR_K(K) \coloneqq BR_K^{1}(K^{-1}) \times \cdots \times BR_{K}^{N_P}(K^{-N_P})
\end{equation*}
denotes the set of jointly-best-response policies. 
This induces a version of the game $\mathcal{G}_{\infty}^{\text{\text{LQ}}}$ in which a solution is now understood as a \textit{policy profile} $K \coloneqq (K^1, \ldots, K^{N_P}) \in \mathcal{L}(\ell_{\infty}^{N_y},\ell_{\infty}^{N_u})$ which is agreeable to all players. 
The solution concept that naturally arises is the generalized \textit{feedback} Nash equilibrium (GFNE).

\begin{definition} \label{def: Feedback_Nash_Equilbrium}
    A profile $K^{\star} = (K^{1^{\star}},\ldots,K^{N_P^{\star}})$ satisfying $K^{\star} \in BR_K(K^{\star})$ is a \textit{generalized feedback Nash equilibrium} (GFNE) of the dynamic game $\mathcal{G}_{\infty}^{\text{\text{LQ}}}$.
\end{definition}

Although GFNEs and GNEs are similar concepts, designing GFNE-seeking routines is considerably more challenging: 
Since Problem \eqref{eq: Best_Response_GFNE} is infinite-dimensional and stochastic, it cannot be approached by the FBS routine from Section \ref{sec: Operator_Splitting}. 
Indeed, even defining a vGFNE is not straightforward, as the strategy space $\mathcal{L}(\ell_{\infty}^{N_y},\ell_{\infty}^{N_u})$ is a Banach space, and therefore it is not equipped with an inner product \cite{Conway2007}. 
In the following, we propose to overcome these issues by considering a specific, but equivalent, representation of the policies $K^p$, for all $p\in\mathcal{P}$. 
We show how this representation can be leveraged to recast the best-response mappings $BR_{K}^p$ ($\forall p\in\mathcal{P}$) as finite-dimensional optimization problems, thus enabling the design of vGFNE-seeking algorithms for solving dynamic games.

% -----------------------------------------------------------------------------
\section{vGFNE-seeking via System Level Synthesis} \label{sec: GFNE_Seeking}

We propose an equivalent representation of control policies, $\{ K^p \}_{p\in\mathcal{P}}$, as system-level responses to noise, $\{ \Phi^p \}_{p\in\mathcal{P}}$.
Using System Level Synthesis (SLS), optimal stabilizing policies can be synthesized through tractable convex optimization problems whose domains are Hilbert spaces \cite{Anderson2019}.
A class of \textit{variational generalized feedback Nash equilibrium} (vGFNE) problems can then be formalized using this equivalent representation. 

In this section, we introduce our SLS-based approach to vGFNE-Seeking in partially observed games. 
We first recast the original best-response mappings $\{ BR_K^p \}_{p\in\mathcal{P}}$ as equivalent \textit{system-level best-response mappings} $\{ BR_{\Phi}^p \}_{p\in\mathcal{P}}$ defined by finite-horizon deterministic convex programs (Section \ref{sec: GNFE_Seeking_BR}). 
The system-level best-response mappings are then used to define a general class of vGFNE-Seeking problems that can be solved through operator splitting, ultimately enabling the solution of the class of dynamic games $\mathcal{G}_{\infty}^{\text{LQ}}$ (Section \ref{sec: vGFNE_Seeking_SLS}). 

\subsection{System-level best-response mappings} \label{sec: GNFE_Seeking_BR}

We start by assuming a stabilizing profile $(K^1, \ldots, K^{N_P})$, ensured by Assumption \ref{as: Game_Is_Stabilisable}. 
Each policy is associated with a transfer matrix $\hat{K}^p \in \mathcal{RH}_{\infty}$, $\hat{K}^p = \sum_{n=0}^{\infty} \frac{1}{z^n} \Phi_{K,n}^p$, defining the feedback policy $\hat{u}^p = \hat{K}^p\hat{y}$ in the frequency domain. 
From the $\mathbb{Z}$-transform of the system $\Sigma_{\mathcal{G}}$ (Eq.~\ref{eq: DynamicGame_SS}), we have
\begin{subequations}
\begin{align}
    z \hat{x}   &= A \hat{x} + B_w \hat{w} + \textstyle\sum_{p\in\mathcal{P}} B_u^p \hat{u}^p;  \label{eq: DynamicGame_SS_Frequency_a}\\
      \hat{u}^p &= \hat{K}^p (C \hat{x} + D_w \hat{w}), \quad \forall p \in \mathcal{P}.      \label{eq: DynamicGame_SS_Frequency_b}
\end{align} \label{eq: DynamicGame_SS_Frequency}%
\end{subequations}
The signals $(\hat{x},\hat{u}^1,\ldots,\hat{u}^{N_P})$ can be posed in terms of $\hat{w}$ as
\begin{equation} 
    \begin{bmatrix}
        \hat{x} \\ \hat{u}^1 \\ \vdots \\ \hat{u}^{N_P}
    \end{bmatrix} 
        = \begin{bmatrix}
            \hat{\Phi}_{xx}       & \hat{\Phi}_{xy}       \\
            \hat{\Phi}_{ux}^{1}   & \hat{\Phi}_{uy}^{1}   \\
            \vdots                     & \vdots                     \\
            \hat{\Phi}_{ux}^{N_P} & \hat{\Phi}_{uy}^{N_P} 
        \end{bmatrix} 
        \begin{bmatrix}
             \hat{\delta}_x \\ \hat{\delta}_y
        \end{bmatrix}, \label{eq: SLS_Parametrization_StateSpace}
\end{equation}
where $\hat{\delta}_x = B_w \hat{w}$ and $\hat{\delta}_y = D_w \hat{w}$, and
\begin{align*}
    \hat{\Phi}_{xx}     &= \big( (zI {-} A) - \textstyle\sum_{p\in\mathcal{P}} B_u^p \hat{K}^p C \big)^{-1};\\
    \hat{\Phi}_{xy}     &= \textstyle\sum_{p\in\mathcal{P}}\hat{\Phi}_{xx} B_u^p \hat{K}^p;\\
    \hat{\Phi}_{ux}^{p} &= \hat{K}^p C \hat{\Phi}_{xx};\\
    \hat{\Phi}_{uy}^{p} &= \hat{K}^p + \textstyle\sum_{\tilde{p}\in\mathcal{P}} \hat{K}^p C \hat{\Phi}_{xx} B_u^{\tilde{p}} \hat{K}^{\tilde{p}}.
\end{align*}
The transfer matrices ($\hat{\Phi}_{xx},\hat{\Phi}_{xy},\hat{\Phi}_{ux}^p,\hat{\Phi}_{uy}^p$) are \textit{system level responses} or \textit{closed-loop maps}: 
They describe how the game's state and the players' actions respond to the noise. 
For ease of notation, we define $\hat{\Phi}_{ux} = (\hat{\Phi}_{ux}^p)_{p\in\mathcal{P}}$ and $\hat{\Phi}_{uy} = (\hat{\Phi}_{uy}^p)_{p\in\mathcal{P}}$. 
Under this representation, the condition that the policy profile $\hat{K} = (\hat{K}^1, \ldots, \hat{K}^{N_P})$ is stabilizing can be expressed directly in terms of the system level response $(\hat{\Phi}_{xx}, \hat{\Phi}_{xy}, \hat{\Phi}_{ux}, \hat{\Phi}_{uy})$. 
Moreover, each policy $\hat{u}^p = \hat{K}^p \hat{y}$ ($\forall p\in\mathcal{P}$) can be represented directly in terms of these closed-loop maps. 
Formally: 

\begin{theorem}[System level parametrization]  \label{thm: SystemLevelSynthesis}
    Consider the system $\Sigma_{\mathcal{G}}$ (Eq.~\ref{eq: DynamicGame_SS_Frequency}) under output-feedback policies $\hat{u}^p = \hat{K}^p\hat{y}$ for all players $p \in \mathcal{P}$. 
    The following statements are true:
    \begin{enumerate}[(a)]
        \item 
        The policy profile $\hat{K} = (\hat{K}^1,\ldots,\hat{K}^{N_P})$ is stabilizing if its system level responses ($\hat{\Phi}_{xx},\hat{\Phi}_{xy},\hat{\Phi}_{ux},\hat{\Phi}_{uy}$) satisfy 
        \begin{subequations}
        \begin{align}
            \begin{bmatrix} zI {-} A ~~ {-}B_u^1 \cdots {-}B_u^{N_P} \end{bmatrix} 
            \begin{bmatrix}
                \hat{\Phi}_{xx} ~~ \hat{\Phi}_{xy} \\
                \hat{\Phi}_{ux} ~~ \hat{\Phi}_{uy}
            \end{bmatrix} 
                &= \begin{bmatrix} I ~~ 0 \end{bmatrix}; \\
            \begin{bmatrix}
                \hat{\Phi}_{xx} ~~ \hat{\Phi}_{xy} \\
                \hat{\Phi}_{ux} ~~ \hat{\Phi}_{uy}
            \end{bmatrix} 
            \begin{bmatrix} zI - A \\ -C \end{bmatrix}
                &= \begin{bmatrix} I \\ 0 \end{bmatrix}; \\
                \hat{\Phi}_{xx},\hat{\Phi}_{xy},\hat{\Phi}_{ux} \in \textstyle\frac{1}{z}\mathcal{RH}_{\infty}, ~~ \hat{\Phi}_{uy} \in&~ \mathcal{RH}_{\infty}.
        \end{align}\label{eq: SLP_AffineSpace}%
        \end{subequations}
        \item 
        The responses $(\hat{\Phi}_{xx},\hat{\Phi}_{xy},\hat{\Phi}_{ux},\hat{\Phi}_{uy})$ satisfying Eq. \eqref{eq: SLP_AffineSpace} are achieved by policies $\hat{K}^p = \hat{\Phi}_{uy}^p - \hat{\Phi}_{ux}^p \hat{\Phi}_{xx}^{-1} \hat{\Phi}_{xy}$ ($p \in \mathcal{P}$) which can be implemented as in Figure \ref{fig: SLP_ControlDiagram}, that is, 
        \begin{subequations}
        \begin{align}
            z\hat{\xi} &= \tilde{\Phi}_{xx}   \hat{\xi} + \tilde{\Phi}_{xy}   \hat{y}; \label{eq: SLP_ControlImplementation_a} \\
             \hat{u}^p &= \tilde{\Phi}_{ux}^p \hat{\xi} +         \Phi_{uy}^p \hat{y}, \label{eq: SLP_ControlImplementation_b}
        \end{align}\label{eq: SLP_ControlImplementation}%
        \end{subequations}
        with ${\tilde \Phi_{xx}}{=}\,z(I {-} z\hat{\Phi}_{xx})$, ${\tilde \Phi_{ux}}^p = z\hat{\Phi}_{ux}^p$, and ${\tilde \Phi_{xy}} = {-}z\hat{\Phi}_{xy}$. 
        Moreover, $\hat{K} = (\hat{K}^1,\ldots,\hat{K}^{N_P})$ is internally stabilizing.
    \end{enumerate}
\end{theorem}
\begin{proof}
    The proof is identical to that of \cite[Theorem 5.1]{Anderson2019} by letting $B_2 = [B_u^1 ~ \cdots ~ B_u^{N_P}]$ and $C_2 = C$.
\end{proof}

\begin{figure}[tb!] \centering
    \includegraphics[width=0.975\columnwidth]{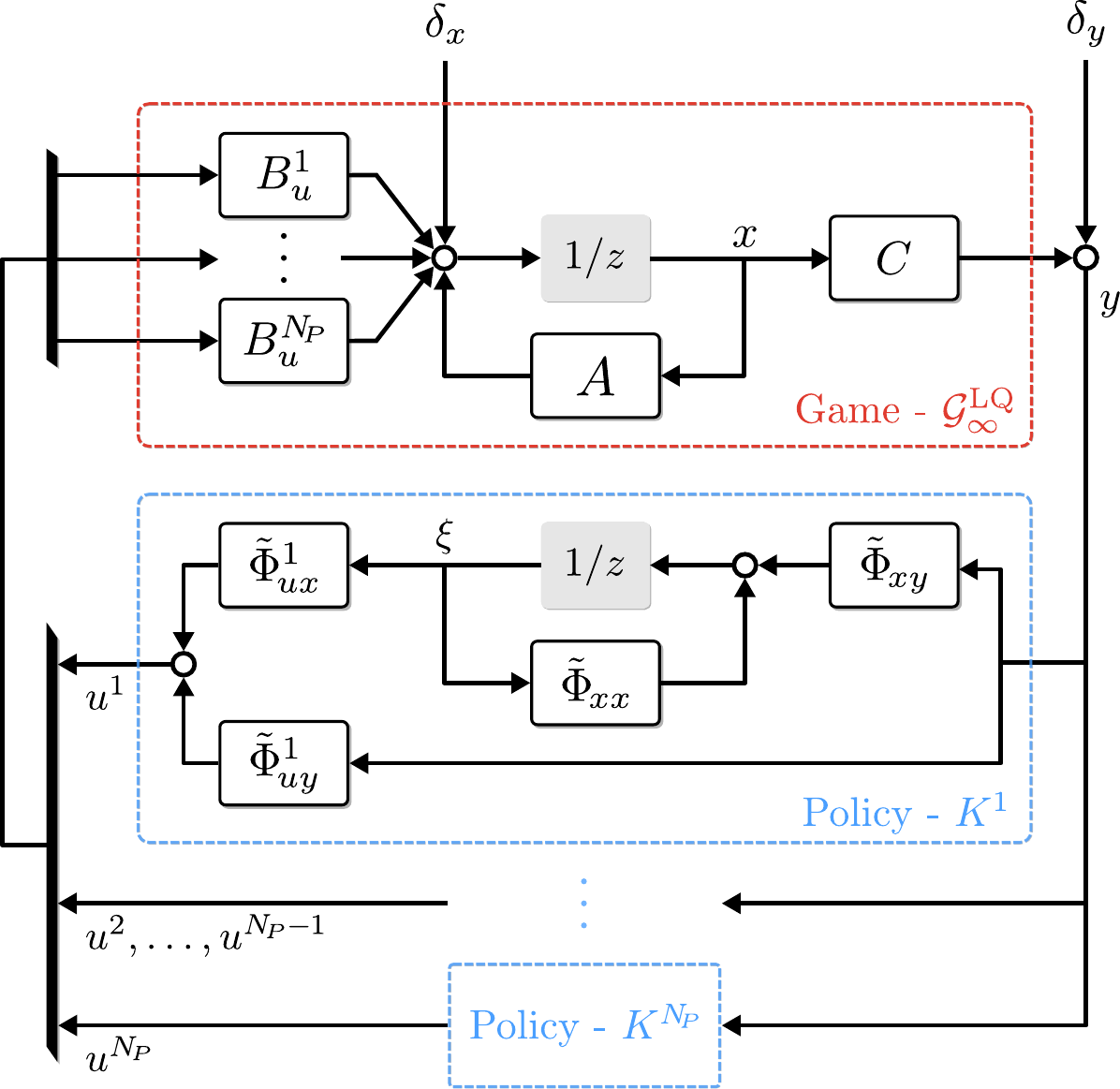}
    
    \caption{The feedback structure for the policies $\hat{K}^p = \hat{\Phi}_{uy}^p - \hat{\Phi}_{ux}^p \hat{\Phi}_{xx}^{-1} \hat{\Phi}_{xy} = {\tilde \Phi_{uy}}^p + {\tilde \Phi_{ux}}^p(zI - {\tilde \Phi_{xx}})^{-1}{\tilde \Phi_{xy}}$, for all $p \in \mathcal{P}$, according to Eq. \eqref{eq: SLP_ControlImplementation}.}
    \label{fig: SLP_ControlDiagram}
\end{figure}

Favouring of a time-domain exposition, we will refer to the system level responses mostly through their spectral factors. 
In this case, Theorem \ref{thm: SystemLevelSynthesis}(a) implies that a stabilizing policy must have its system level responses satisfying the system
\begin{subequations}
    \begin{empheq}[left=\Sigma_{\Phi}:\empheqlbrace]{align} 
        \Phi_{xx,n+1} &= A \Phi_{xx,n} + \textstyle\sum_{p\in\mathcal{P}} B_u^p \Phi_{ux,n}^p; \label{eq: DynamicGame_SLS_Primal_a} \\
        \Phi_{xy,n+1} &= A \Phi_{xy,n} + \textstyle\sum_{p\in\mathcal{P}} B_u^p \Phi_{uy,n}^p, \label{eq: DynamicGame_SLS_Primal_b}
    \end{empheq} \label{eq: DynamicGame_SLS_Primal}%
\end{subequations}
and the ``dual'' system
\begin{subequations}
    \begin{empheq}[left=\Sigma_{\Phi}^*:\empheqlbrace]{align} 
        \Phi_{xx,n+1} &= \Phi_{xx,n} A + \Phi_{xy,n} C; \label{eq: DynamicGame_SLS_Dual_a} \\
        \Phi_{ux,n+1}^{p} &= \Phi_{ux,n}^{p} A + \Phi_{uy,n}^{p} C, ~~ (\forall p \in \mathcal{P}), \label{eq: DynamicGame_SLS_Dual_b}
    \end{empheq} \label{eq: DynamicGame_SLS_Dual}%
\end{subequations}
with $\Phi_{xx,1} = I_{N_x}$, and $\Phi_{xx,0} = \Phi_{xy,0} = \Phi_{ux,0}^p = 0$ ($\forall p \in \mathcal{P}$) due to strict causality. 
A time-domain characterization of each policy $\hat{K}^p$ can also be obtained in terms of these components by applying an inverse $\mathbb{Z}$-transform to Eq. \eqref{eq: SLP_ControlImplementation}:
\begin{corollary} \label{cor: ControlImplementation}
    The policy $\hat{K}^p$ from Eq. \eqref{eq: SLP_ControlImplementation} is defined by the kernel $\Phi^p_K = \Phi_{uy}^p {-} \Phi_{ux}^p {*} \Phi_{xx}^{-1} {*} \Phi_{xy}$ and is implemented as
    \begin{subequations}
        \begin{align}
            \xi_{t+1} &= -          \textstyle\sum_{n=0}^{t} \Phi_{xx,n{+}2}   \xi_{t-n} - \textstyle\sum_{n=0}^{t} \Phi_{xy,n{+}1} y_{t-n}; \\
              u^p_{t} &= \phantom{-}\textstyle\sum_{n=0}^{t} \Phi_{ux,n{+}1}^p \xi_{t-n} + \textstyle\sum_{n=0}^{t} \Phi_{uy,n}^p   y_{t-n}.
        \end{align} \label{eq: SLP_ControlImplementation_Time}%
    \end{subequations}
    using an auxiliary \emph{internal state} $\xi = (\xi_n)_{n\in\mathbb{N}}$ with $\xi_0 = 0$.
\end{corollary}
The system level parametrization enables a methodology for policy synthesis consisting of searching the space of stabilizing policies directly through $(\Phi_{xx}, \Phi_{xy}, \Phi_{ux}^p, \Phi_{uy}^p)$, for all $p \in \mathcal{P}$. 
In this direction, consider that the players plan their actions by choosing a desired system level response $\Phi_u^p = (\Phi_{ux}^p, \Phi_{uy}^p)$ to noise, then implementing their policy as in Corollary~\ref{cor: ControlImplementation}. 
The original best-response mappings $\{ BR_K^p \}_{p\in\mathcal{P}}$ are then equivalent to \textit{system-level best-response mappings} $\{ BR_{\Phi}^p \}_{p\in\mathcal{P}}$ corresponding to the solution set of problems of the form
\begin{subequations}
\begin{align}
    \minimize_{\Phi_{ux}^p, \Phi_{uy}^p} \quad
         & \sum_{n=0}^{\infty} \left\| \begin{bmatrix} W^p_x & W^p_u \end{bmatrix}\hspace*{-0.25em}\begin{bmatrix} \Phi_{xx,n} & \Phi_{xy,n} \\ \Phi_{ux,n}^p & \Phi_{uy,n}^p \end{bmatrix} \begin{bmatrix} B_w \\ D_w \end{bmatrix} \right\|_F^2  \label{eq: Best_Response_GFNE_SLS_a}\\
    \subjectto_{n = 0, 1, \ldots} \quad
        & \Sigma_{\Phi} \text{ (Eq.~\ref{eq: DynamicGame_SLS_Primal})}, ~~\Sigma_{\Phi}^* \text{ (Eq.~\ref{eq: DynamicGame_SLS_Dual})} , \label{eq: Best_Response_GFNE_SLS_b}\\[-1.5ex]
        & \begin{bmatrix} \Phi_{xx,n} & \Phi_{xy,n} \\ \Phi_{ux,n}^p & \Phi_{uy,n}^p \end{bmatrix} \in \mathcal{S}_{\Phi,n}^p, \label{eq: Best_Response_GFNE_SLS_c} \\
        & \begin{bmatrix} G_x & G_u \\ & \widetilde{G}_{u^p} \end{bmatrix}\hspace*{-0.25em}\begin{bmatrix} \Phi_{xx} & \Phi_{xy} \\ \Phi_{ux} & \Phi_{uy} \end{bmatrix}\begin{bmatrix} B_w \\ D_w \end{bmatrix} * w_n \preceq 1, \label{eq: Best_Response_GFNE_SLS_d}
\end{align} \label{eq: Best_Response_GFNE_SLS}%
\end{subequations}%
where we consider $\widetilde{G}_{u^p} = e_p^{\tran} \otimes G_{u^p}$ to simplify its statement. 
The objective (Eq. \ref{eq: Best_Response_GFNE_SLS_a}) and operational constraints (Eq. \ref{eq: Best_Response_GFNE_SLS_d}) are obtained from the definition of the system level responses. 
The original structural constraints (Eq. \ref{eq: Best_Response_GFNE_c}) are translated into 
(i) the system level parametrization (Eq. \ref{eq: Best_Response_GFNE_SLS_b}) enforcing internal stability, and 
(ii) the system level constraints (Eq. \ref{eq: Best_Response_GFNE_SLS_c}) enforcing $\mathcal{G}_{\infty}^{\text{\text{LQ}}}$-related restrictions via the policy's system-level parameters.
Note that player's policies are linked through the $(\Phi_{xx}, \Phi_{xy})$ responses. 
We refer to \cite{Anderson2019} for more details on how Problem \eqref{eq: Best_Response_GFNE_SLS} can be derived from Problem \eqref{eq: Best_Response_GFNE}.  

The mappings $\{ BR_{\Phi}^p \}_{p\in\mathcal{P}}$ are still intractable, as Problem \eqref{eq: Best_Response_GFNE_SLS} is also infinite-dimensional and stochastic. 
In this case, however, it can be made tractable by a specific choice of the structural constraints (Eq.~\ref{eq: Best_Response_GFNE_SLS_c}) and a robust formulation of the operational constraints (Eq.~\ref{eq: Best_Response_GFNE_SLS_d}); as shown in the following. 

% -----------------------------------------------------------------------------
\vskip0.5em
\subsubsection*{Structural constraints} 

The problems in $\{ BR_{\Phi}^p \}_{p\in\mathcal{P}}$ can be made finite-dimensional by restricting the choice of $K^p$ to policies with finite-impulse response (FIR) kernels. 
Using the sets $\mathcal{S}_{\Phi}^p = (\mathcal{S}_{\Phi,n}^p)_{n\in\mathbb{N}}$, this can be done by enforcing %
\begin{equation} \label{eq: FIR_Constraints_Sparsity_A}
    \begin{bmatrix} 
        \Phi_{xx,n}   & \Phi_{xy,n}   \\ 
        \Phi_{ux,n}^p & \Phi_{uy,n}^p 
    \end{bmatrix} 
    = 0, \quad \text{for } n \geq N.
    \\
\end{equation}
In practice, this corresponds to enforcing the terminal conditions $\Phi_{xx,N} = \Phi_{xy,N} = \Phi_{ux,N}^p = \Phi_{uy,N}^p = 0$, then restricting Problem \eqref{eq: Best_Response_GFNE_SLS} to the finite sequences of spectral factors $(\Phi_{xx,n}, \Phi_{xy,n}, \Phi_{ux,n}^p)_{n=1}^{N-1}$ and $(\Phi_{uy,n}^p)_{n=0}^{N-1}$. 
Additionally, FIR kernels simplify the implementation of $K^p$ by fixing the total number of summand terms in Corollary \ref{cor: ControlImplementation}.

\begin{remark}
    The terminal constraints induced by Eq.~\eqref{eq: FIR_Constraints_Sparsity_A} cannot be satisfied when the system is not controllable and observable. 
    While they can be relaxed, the stability and optimality properties of the resulting controller are well-understood only for the state-feedback case \cite{Anderson2019, Dean2020}. 
\end{remark}

The remaining constraint sets $(\mathcal{S}_{\Phi,n}^p)_{n=0}^N$ are designed to impose structure onto $K^p$ through their system-level parametrization (Corollary \ref{cor: ControlImplementation}). 
We consider actuation and communication delays encoded by the sparsity constraints
\begin{equation} \label{eq: Structural_Constraints_SLP}
        \mathrm{Sp}
    \begin{bmatrix} 
        \Phi_{xx,n}   & \Phi_{xy,n}   \\ 
        \Phi_{ux,n}^p & \Phi_{uy,n}^p 
    \end{bmatrix} 
    =
    \mathrm{Sp}
    \begin{bmatrix} 
                     A^{\alpha_n} &              A^{\alpha_n}    C^{\tran} \\ 
        B_u^{p^\tran}A^{\alpha_n} & B_u^{p^\tran}A^{\alpha_{n+1}}C^{\tran} 
    \end{bmatrix},
\end{equation}
for $n = 0,\ldots,N$, with $\alpha_n = \max(0, \lfloor (n-d_a)/d_c \rfloor)$ and $\mathrm{Sp}[\cdot]$ denoting the support of a matrix. 
Under this setup, the policies $K^p$ ($p\in\mathcal{P}$) satisfy that the $i$-th state component $[x]_i = ([x_t]_i)_{t\in\mathbb{N}}$ is only affected by the $i$-th action component $[B_u u^p]_i = ([B_u u^p_t]_i)_{t\in\mathbb{N}}$ after an \textit{actuation delay} of $d_a \geq 0$ steps. 
Additionally, it specifies that measurements propagate in the communication network (which has the same topology as the physical system) with a \textit{communication delay} of $d_c > 0$. 
These parameters are often determined by the hardware of the underlying control system and communication infrastructure. 
We refer to policies satisfying this information pattern as \textit{$(d_a,d_c)$-delayed feedback policies}.
This choice is not restrictive and the results in this paper generalize to a broader class of informational constraints (such as those discussed in \cite{Wang2017}). 
Finally, we note that the ability to impose sparsity to feedback policies, a central feature of the SLS framework, enables the solution of GFNE problems with asymmetric information patterns; a major challenge in the field \cite{Nayyar2012, Basar2014}.

% -----------------------------------------------------------------------------
\vskip0.5em
\subsubsection*{Operational constraints}

The mapping $BR_{\Phi}^p$ is stochastic due to the original operational constraints (Eq.~\ref{eq: Best_Response_GFNE_d}) becoming random inequalities in the system-level problem (Eq.~\ref{eq: Best_Response_GFNE_SLS_d}). 
The problem can be made deterministic by instead enforcing
\begin{equation} \label{eq: SLS_Chance_Constraint}
    \mathbf{prob}\left( \begin{bmatrix} G_x & G_u \\ & \widetilde{G}_{u^p} \end{bmatrix}_{i,:} \begin{bmatrix} \Phi_{xx} & \Phi_{xy} \\ \Phi_{ux} & \Phi_{uy} \end{bmatrix}\begin{bmatrix} B_w \\ D_w \end{bmatrix} * w_n \leq 1 \right) \geq \rho,
\end{equation}
for every $i$-th row of the matrix $G^p = \mathrm{col}([G_x ~ G_u], [0 ~ \widetilde{G}_{u^p}])$, given a probability $\rho \in (0.5, 1)$. 
Since $w$ is a white-noise process and $\Phi_w = (\Phi_{xx}B_w {+} \Phi_{xy}D_w,~ \Phi_{ux}B_w {+} \Phi_{uy}D_w)$ is FIR, Eq. \eqref{eq: SLS_Chance_Constraint} can be expressed as 
$
    \mathbf{prob}\big( G_{\Phi,i}^{p} \tilde{w} \leq 1 \big) \geq \rho
$
given the random vector $\tilde{w}$ with $\mathrm{E}\tilde{w} = 0$ and $\mathrm{E}(\tilde{w}\tilde{w}^{\tran}) = I_{NN_w}$, and the block-matrix $G_{\Phi,i}^p = \big[ [G^p]_{i,:}\Phi_{w,0}~\cdots~[G^p]_{i,:}\Phi_{w,N} \big]$. 
The constraints Eq. \eqref{eq: Best_Response_GFNE_SLS_d} can thus be realized using the cumulative distribution function of each $G_{\Phi,i}^{p} \tilde{w}$ ($i = 1, \ldots, N_{G^p}$). 

A common assumption is to consider the random vector $\tilde{w}$ to be Gaussian, i.e., $\tilde{w} \sim \mathrm{Normal}(0,I)$. 
Using standard results from optimization theory \cite{Boyd2004}, and expanding $G_{\Phi,i}^p$, Eq.~\eqref{eq: SLS_Chance_Constraint} is then equivalent to the second-order conic (SOC) constraint
\begin{equation} \label{eq: SLS_Chance_Constraint_SOC}
    \left\| \left( \begin{bmatrix} G_x & G_u \\ & \widetilde{G}_{u^p} \end{bmatrix}_{i,:}\hspace*{-0.25em}\begin{bmatrix} \Phi_{xx} & \Phi_{xy} \\ \Phi_{ux} & \Phi_{uy} \end{bmatrix}\hspace*{-0.25em} \begin{bmatrix} B_w \\ D_w \end{bmatrix} \right)^{\tran}  \right\|_{\ell_2} \leq \frac{1}{Q(\rho)}
\end{equation}%
with $Q : [0,1] \to \mathbb{R}$ being the quantile function of the standard Normal distribution (which is only positive for $\rho > 0.5$).
We remark that these constraints cannot ensure that the synthesized policy satisfies the original $U^p$ for all realizations of the noise $w$; even without the Gaussian assumption. 
However, at the risk of conservativeness, the best-response mapping $BR_{\Phi}^p$ can still be designed to ensure with arbitrarily high probability $\rho > 0.5$ that such constraints will be satisfied during operation. 

% --------------------------------------------------------------------
\subsection{vGFNE-seeking via SLS and FBS} \label{sec: vGFNE_Seeking_SLS}

In this section, to simplify notation, we momentarily define $\Phi_{x} = [\Phi_{xx} ~ \Phi_{xy}]$, $\Phi_u^p = [\Phi_{ux}^p ~ \Phi_{uy}^p]$, and $W_w = (B_w, D_w)$.
Due to the equivalence between $BR_{K}$ and $BR_{\Phi}$, it is clear that a stabilizing policy profile is a GFNE (i.e., $K^{\star} \in BR_K(K^{\star})$) if its associated system level response is a fixed-point of $BR_{\Phi}$ (i.e., $\Phi_{u}^{\star} \in BR_{\Phi}(\Phi_{u}^{\star})$). 
The concept of a \emph{variational generalized feedback Nash equilibrium} can be defined via the system level parametrization of the output-feedback policies: 
A policy profile $K^{\star} = (\hat{\Phi}_{uy}^{p^{\star}} - \hat{\Phi}_{ux}^{p^{\star}} (\hat{\Phi}_{xx}^{\star})^{-1} \hat{\Phi}_{xy}^{\star})_{p\in\mathcal{P}}$ is a vGFNE when the response profile $\Phi_u^{\star} = (\Phi_u^{1^{\star}}, \ldots, \Phi_u^{N_P^{\star}})$ is a vGNE of the (static) game defined by the objectives
\begin{equation*}
    J_{\Phi}^p(\Phi_u^{p}, \Phi_u^{-p}) = \sum_{n=0}^{N} \Big\| \begin{bmatrix} W^p_x & W^p_u \end{bmatrix}\hspace*{-0.25em}\begin{bmatrix} \Phi_{x,n} \\ \Phi_{u,n}^p \end{bmatrix} W_w \Big\|_F^2
\end{equation*}
and the finite-dimensional and convex global feasible set%
\begin{align*}
    \mathcal{U}_{\Phi,\mathcal{G}} = \{& \Phi_u \in \ell_2[0,N] \mid \\
        &~ \Sigma_{\Phi} \text{ (Eq.~\ref{eq: DynamicGame_SLS_Primal})}, ~~\Sigma_{\Phi}^* \text{ (Eq.~\ref{eq: DynamicGame_SLS_Dual})}, ~~\Phi_{x} = \Phi_{u} = 0, \\
        &~ \mathrm{Sp}\begin{bmatrix} \Phi_{x,n} \\ \Phi_{u,n} \end{bmatrix} = \mathrm{Sp} \begin{bmatrix} A^{\alpha_n} & A^{\alpha_n}C^{\tran} \\ B_u^{\tran}A^{\alpha_n} & B_u^{\tran}A^{\alpha_{n+1}}C^{\tran} \end{bmatrix} ~ (\forall n), \\
        &~ \Big\| \Big( \begin{bmatrix} G_x & G_u \\ & \widetilde{G}_u \end{bmatrix}_{i,:}\begin{bmatrix} \Phi_{x} \\ \Phi_{u} \end{bmatrix} W_w \Big)^{\tran} \Big\|_{\ell_2} \leq \frac{1}{Q(\rho)}~~ (\forall i) \Big\},
\end{align*}
with $\widetilde{G}_u = \mathrm{blkdiag}(G_{u^1}, \ldots, G_{u^{N_P}})$.
Since $\Phi_{x} = (\Phi_{x,n})_{n=1}^N$ and $\Phi_{u}^p = (\Phi_{u,n}^p)_{n=0}^N$ are FIR, these signals can be represented as (concatenated) matrices.
The vGFNE-seeking problem can thus be formulated as the monotone inclusion problem
\begin{multline} \label{eq: vGFNE_Monotone_Inclusion_Problem}
    \text{Find } \Phi_u^{\star} \in \mathbb{R}^{(N+1)N_u \times (N_x+N_y)} \\
        \text{ such that } 0 \in F_{\Phi}(\Phi_u^{\star}) + N_{\mathcal{U}_{\Phi,\mathcal{G}}}(\Phi_u^{\star}),
\end{multline}
given the pseudo-gradient $F_{\Phi}(\Phi_u) = (\nabla_{\Phi_u^p} J_{\Phi}^p(\Phi_u^p, \Phi_u^{-p}))_{p\in\mathcal{P}}$. 
Note that $F_{\Phi}$ is maximal monotone because its components, $\nabla_{\Phi_u^p} J_{\Phi}^p$ ($\forall p \in \mathcal{P}$), are the gradients of closed convex proper functions \cite{Bauschke2017}.
The Problem \eqref{eq: vGFNE_Monotone_Inclusion_Problem} is a special instance of the monotone inclusion problem presented in Section \ref{sec: Preliminaries}, and thus it can be solved via FBS.
The resulting routine (Algorithm \ref{alg: FB_vGFNE-Seeking_SLS}) will be referred to as the \textit{vGFNE-Seeking via FBS} method\footnotemark.

\footnotetext{The update-index $k \in \mathbb{N}$ should not be mistaken with time-indices $n \in \mathbb{N}$. In particular, $\Phi_{n|k}$ is the $n$-th factor of a kernel $\Phi \in \ell_2(\mathbb{N})$ after $k$ updates.}

\begin{algorithm}[tb!] \label{alg: FB_vGFNE-Seeking_SLS}
    \caption{vGFNE-Seeking via FBS}
    Initialize $\Phi_{u|0} = (\Phi_{u|0}^1, \ldots, \Phi_{u|0}^{N_P})$\;
    \For{$k = 0, 1, 2, \ldots$}{ 
        \For{each player $p \in \mathcal{P}$}{
            $\hat{K}_k^p \coloneqq \hat{\Phi}_{uy|k}^p - \hat{\Phi}_{ux|k}^p \hat{\Phi}_{xx|k}^{-1} \hat{\Phi}_{xy|k}$\;
            $\Phi_{u|k+1/2}^p \coloneqq \Phi_{u|k}^p - \eta \nabla_{\Phi_u^p} J_{\Phi}^p(\Phi_{u|k}^p, \Phi_{u|k}^{-p})$\;
        }
        \For{coordinator}{
            $\Phi_{u|k+1} \coloneqq \proj_{\mathcal{U}_{\Phi,\mathcal{G}}}(\Phi_{u|k+1/2})$\;
        }
    }
\end{algorithm} 

The distinct feature in this vGNE-seeking routine is the policy update (line 4): 
In practice, it corresponds to plugging ($\Phi_{xx|k}, \Phi_{xy|k}, \Phi_{ux|k}^p, \Phi_{uy|k}^p$) on the policy's implementation from Corollary \ref{cor: ControlImplementation}. 
Importantly, Algorithm \ref{alg: FB_vGFNE-Seeking_SLS} does not require knowledge of the state $x$ or input $u^p$ signals. 
A direct consequence is that the players can perform this policy learning routine while simultaneously operating the underlying networked system. 
Specifically, the players can operate according to the policies $K_k = (K_k^{1}, \ldots, K_k^{N_P})$ until the coordinator processes their proposals $\Phi_{u|k+1/2} = (\Phi_{u|k+1/2}^p)_{p\in\mathcal{P}}$ and then returns the responses $\Phi_{u|k+1} = (\Phi_{u|k+1}^p)_{p\in\mathcal{P}}$ which they can use to update their policies. 
This enables vGFNE-seeking in games that cannot be interrupted for players to redesign their policies. 

We finish the section with discussions on the convergence and computational properties of Algorithm \ref{alg: FB_vGFNE-Seeking_SLS}.

% -----------------------------------------------------------------------------
\vskip0.5em
\subsubsection*{Convergence properties}

From Lemma \ref{thm: fFB_Convergence}, the algorithm converges to a unique solution if the learning rate satisfies $\eta \in (0, 2M_{F_{\Phi}}/L_{F_{\Phi}}^2)$, with $L_{F_{\Phi}}$ and $M_{F_{\Phi}}$ being the Lipschitz and strong-monotonicity constants of $F_{\Phi}$. 
Given the quadratic objectives $J^p_{\Phi}$ ($p \in \mathcal{P}$) and linear dynamics $\Sigma_{\Phi}$ (Eq.~\ref{eq: DynamicGame_SLS_Primal}), each player's gradient can be shown to have the affine form 
\begin{equation} \label{eq: Pseudogradient_Affine_Players}
    \nabla_{\Phi_u^p}J_{\Phi}^p(\Phi_u) = 2\textstyle\sum_{\tilde{p}\in\mathcal{P}}(H_{\Phi}^{pp^\tran} H_{\Phi}^{p\tilde{p}}) \Phi_u^{\tilde{p}} (W_w W_w^{\tran}) + h^p
\end{equation} given the block-triangular matrices $H_{\Phi}^{p\tilde{p}} = [H_{\Phi,n,n'}^{p\tilde{p}}]_{0 \leq n,n' \leq N}$,
\begin{equation} \label{eq: Pseudogradient_Hessian_Players}
    H_{\Phi,n,n'}^{p\tilde{p}} = \begin{cases}
        W_u^p & \text{if } n = n' \text{ and } p = \tilde{p} \\
        W_x^p A^{(n-1)-n'} B_u^{\tilde{p}} & \text{if } n > n' \\
        0 & \text{otherwise}
    \end{cases}
\end{equation}
and the appropriate constant matrix $h^p \in \mathbb{R}^{NN_u \times (N_x+N_y)}$.
As such, also the pseudo-gradient $F_{\Phi}(\Phi_u)$ must be an affine operator.
We can then establish the following result:
\begin{theorem} \label{thm: FPhi_Lipschitz_Monotone}
    The system-level pseudo-gradient $F_{\Phi}$ is $M_{F_{\Phi}}$-strongly-monotone and $L_{F_{\Phi}}$-Lipschitz with constants
    \begin{subequations}
    \begin{align}
        M_{F_{\Phi}} &= \lambda_{\min}\big( D_{\Phi}^{\tran} H_{\Phi} + H_{\Phi}^{\tran} D_{\Phi} \big) \sigma_{\min}^2(W_w), \\
        L_{F_{\Phi}} &= 2\sigma_{\max}(D_{\Phi}^{\tran}H_{\Phi})\sigma^2_{\max}(W_w).
    \end{align}\label{eq: FPhi_Lipschitz_Monotone}%
    \end{subequations}
    given $D_{\Phi} = \mathrm{blkdiag}(H_{\Phi}^{pp})_{p\in\mathcal{P}}$ and $H_{\Phi} = [H_{\Phi}^{p\tilde{p}}]_{p,\tilde{p}\in\mathcal{P}}$.
\end{theorem}
\begin{proof}
    See the Appendix.
\end{proof}
Interestingly, this implies that the pseudo-gradient $F_{\Phi}$ is ill-conditioned if so is the noise-filtering matrix $W_w = (B_w, D_w)$ and that convergence rates are not affected by the noise if $W_w = I$.
Moreover, it suggests that open-loop unstable games (when $\rho(A) \geq 1$) might require a careful tuning of $W_x^p$ ($\forall p$) to avoid the terms $W_x^p A^n B_u^{\tilde{p}}$ ($n = 1,\ldots,N{-}1$) in Eq.~\eqref{eq: Pseudogradient_Hessian_Players} from exploding, then affecting the condition number of $F_{\Phi}$.

Applying Lemma \ref{thm: fFB_Convergence} and Theorem \ref{thm: FPhi_Lipschitz_Monotone}, the following holds:
\begin{corollary} \label{thm: vGFNE_FB_Convergence}
    If $\eta \in (0, 2M_{F_\Phi}/L_{F_\Phi}^2)$, with ($M_{F_\Phi}$,$L_{F_\Phi}$) as in Eq.~\eqref{eq: FPhi_Lipschitz_Monotone}, then Algorithm \ref{alg: FB_vGFNE-Seeking_SLS} converges linearly to the unique vGFNE $\hat{K}^{\star} = (\hat{\Phi}_{uy}^{p^{\star}} - \hat{\Phi}_{ux}^{p^{\star}} \hat{\Phi}_{xx}^{\star^{-1}} \hat{\Phi}_{xy}^{\star})_{p\in\mathcal{P}}$, parametrized by the vGNE $\Phi_u^{\star} \in \mathbf{VI}(BR_{\Phi})$, with rate
    \begin{equation} 
        \lim_{k \to \infty} \frac{\| \Phi_{u|k+1} - \Phi_u^{\star} \|_{F}}{\| \Phi_{u|k} - \Phi_u^{\star} \|_{F}} = \sqrt{1 - \eta (2M_{F_\Phi} - \eta L_{F_\Phi}^2)}
    \end{equation}
    from any initial $\Phi_{u|0} \in \mathbb{R}^{(N+1)N_u \times (N_x+N_y)}$. 
\end{corollary}
In summary, whenever the (system-level) pseudo-gradient $F_{\Phi}$ has a condition number $\kappa_{F_{\Phi}} = L_{F_{\Phi}}/M_{F_{\Phi}} \in (0,\infty)$, then the vGFNE-Seeking via FBS method converges linearly to the unique solution to the monotone inclusion Problem \eqref{eq: vGFNE_Monotone_Inclusion_Problem}. 

% -----------------------------------------------------------------------------
\vskip0.5em
\subsubsection*{Computation} 

Numerically, the pseudo-gradient operator $F_{\Phi}$ can be evaluated through automatic differentiation and the projection operator $\proj_{\mathcal{U}_{\Phi,\mathcal{G}}}$ by directly solving the associated optimization problem. 
However, the structure of this vGFNE problem allows for more efficient numerics. 
In the case of $F_{\Phi} = (\nabla_{\Phi_{u}^p} J^p_{\Phi})_{p\in\mathcal{P}}$, each component can be computed as 
\begin{multline}
    \nabla_{\Phi_{u}^p} J^p(\cdot) 
    = \Big( 2({W_u^p}^{\tran}W_u^p)\Phi_{u,n|k}^p(W_w W_w^{\tran}) \\ + 2({B_u^p})^{\tran} \Delta^p_{x,n|k}(W_w W_w^{\tran}) \Big)_{n=0}^N,
\end{multline}
with the sensitivities $\Delta^p_{x|k} = ({\Delta^p_{x,n|k}})_{n=0}^N$ obtained by first forward-propagating the $\Phi_{x|k}$ responses,
\begin{align*} 
    \Phi_{x,n+1|k} &= A \Phi_{x,n|k} + \textstyle\sum_{p\in\mathcal{P}} B_u^p \Phi_{u,n|k}^p,
\end{align*}
from $\Phi_{x,0|k} = [I_{N_x} ~~ 0]$, 
then by backward-propagating
\begin{align*}
    \Delta_{x,n-1|k}^p &= A^{\tran} \Delta_{x,n|k}^p + ({W_x^p}^{\tran} W_x^p) \Phi_{x,n|k}
\end{align*}
from $\Delta_{x,N|k} = 0$. 
These operations are not demanding and thus players are not required large computational resources to participate in this vGFNE-seeking routine. 
Moreover, the gradient $\nabla_{\Phi_{u}^p} J^p_{\Phi}(\cdot)$ depends on other players' $\Phi_{u|k}^{-p}$ only through $\Phi_{x|k}$; 
when this response is already available (e.g., provided by the coordinator), then each $p$-th player is capable of computing its proposed update $\Phi^p_{u|k+1/2}$ using only private information. 

In general, the projection map
\begin{equation} \label{eq: vGFNE_Projection_Operator}
    \proj_{\mathcal{U}_{\Phi,\mathcal{G}}}(\cdot) = \argmin_{\Phi_u} \big\{ \| \Phi_u - \Phi_{u|k+1/2} \|_{\ell_2} \mid \Phi_u \in \mathcal{U}_{\Phi,\mathcal{G}} \big\}
\end{equation}
is a sparse convex optimization that can be solved efficiently. 
However, the problem has $(N{+}1)N_u(N_x{+}N_y)$ decision variables and $NN_x(N_x{+}N_y)$ equality constraints, and thus scales poorly with the state-space dimensions $(N_x,N_u,N_y)$ and FIR horizon $N$. 
Its computational burden can be alleviated by exploiting the structure of Problem \eqref{eq: vGFNE_Projection_Operator} and using a distributed approach for its solution. 
Consider the problem
\begin{equation} \label{eq: vGFNE_Projection_Operator_Separable}%
    \minimize_{\Psi, \Lambda} ~ J^{(r)}_{\Phi}(\Psi) + J^{(c)}_{\Phi}(\Lambda)  ~\subjectto~ \Psi = \Lambda
\end{equation}%
defined by the extended-real-value functions
\begin{align*}
    J^{(r)}_{\Phi}(\Psi) &= \begin{cases}
        \frac{1}{2} \| \Psi - \Phi_{u,k+1/2} \|_{\ell_2}^2 & \text{if Eq.~(\ref{eq: DynamicGame_SLS_Dual}, \ref{eq: FIR_Constraints_Sparsity_A}) holds} \\
        \infty & \text{otherwise} \\
    \end{cases}
    \\
    J^{(c)}_{\Phi}(\Lambda) &= \begin{cases}
        \frac{1}{2} \| \Lambda - \Phi_{u,k+1/2} \|_{\ell_2}^2 & \text{if Eq.~(\ref{eq: DynamicGame_SLS_Primal}, \ref{eq: FIR_Constraints_Sparsity_A}, \ref{eq: SLS_Chance_Constraint_SOC}) holds} \\
        \infty & \text{otherwise} \\
    \end{cases}
\end{align*}
The Problem \eqref{eq: vGFNE_Projection_Operator_Separable} can be solved by the Alternating Direction Method of Multipliers (ADMM, \cite{Ryu2022}), and either $\Psi^{\star}$ or $\Lambda^{\star}$ used as the projection $\Phi_{u|k+1}$. 
The subproblem associated with $J^{(r)}_{\Phi}$ considers only the dual dynamic constraints $\Sigma_{\Phi}^*$ (Eq.~\ref{eq: DynamicGame_SLS_Dual}) and the structural constraints from Eq. \eqref{eq: FIR_Constraints_Sparsity_A}. 
The subproblem associated with $J^{(c)}_{\Phi}$ considers the primal dynamic constraints $\Sigma_{\Phi}$ (Eq.~\ref{eq: DynamicGame_SLS_Primal}), the structural constraints from Eq. \eqref{eq: FIR_Constraints_Sparsity_A}, and the operational constraints from Eq. \eqref{eq: SLS_Chance_Constraint_SOC}. 
For many cases (e.g., diagonal weighting matrices and bound constraints), these subproblems are, respectively, row-wise and column-wise separable and can be reduced into smaller problems to be solved in parallel. 
This can lead to a substantial performance improvement, making Algorithm \ref{alg: FB_vGFNE-Seeking_SLS} scalable to large-scale problems. 
We refer to \cite{Wang2018} for more details on partially separable SLS problems and their solution via ADMM.

% -----------------------------------------------------------------------------
\section{Illustrative application: Stabilization of a partially observed power grid} \label{sec: Examples}

In the following, we study the efficacy and performance of the proposed vGFNE-Seeking via FBS (Algorithm \ref{alg: FB_vGFNE-Seeking_SLS}) in a simulated application inspired by the decentralized control of power networks. 
The example is adapted from \cite{Wang2018}.

We consider a $3 \times 3$ grid network of interconnected subsystems (Figure \ref{fig: PowerNetwork}). 
Each $p$-th node represents a power system operated by a self-interested agent aiming to stabilize its \textit{phase angle deviation} $\theta^p$ and \textit{frequency deviation} $\dot{\theta}^p$ against disturbances. 
For this purpose, each agent has actuators capable of applying a load directly to $\dot{\theta}^p$. 
However, for budgetary reasons, each subsystem is equipped only with a phase measurement unit producing noisy readings of ${\theta}$.
The measurements are publicly available through some communication network (which is assumed to have the same topology as in Figure \ref{fig: PowerNetwork}).

Each $p$-th subsystem has the continuous-time dynamics
\begin{equation} \label{eq: SwingDynamics}
    m^p \ddot{\theta}^p + d^p \dot{\theta}^p = -\textstyle\sum_{\tilde{p}\in\mathcal{P}} \kappa^{p,\tilde{p}}(\theta^p - \theta^{\tilde{p}}) + u^p + \delta_{x}^p,
\end{equation} 
where $u^p$ and $\delta_x^p$ represent the controllable load and the external disturbances, respectively. 
The fixed parameters are sampled from $m^p \sim \mathrm{Uniform}(0.5,1)$,  $d^p \sim \mathrm{Uniform}(1,1.5)$, and  $k^{p,\tilde{p}} = k^{\tilde{p},p} \sim \mathrm{Uniform}(0.5,1)$ with $k^{p,\tilde{p}} = 0$ if the link $j \to i$ is not present in the topology shown in Figure \ref{fig: PowerNetwork}. 
The phase measurements of each $p$-th subsystem are represented by the signal $y^p = \theta^p + \delta_y^p$, with $\delta_y^p$ being the measurement noise.
The disturbances $\delta_x^p$ and the measurement noise $\delta_y^p$ are modelled as additive white Gaussian noise (AWGN) with standard deviations of $1$ and $0.1$, respectively. 

\begin{figure}[htb!] \centering
    \includegraphics[width=0.975\columnwidth]{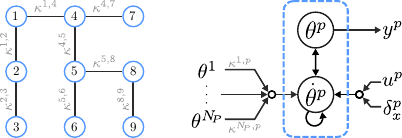}

    \caption{Power grid: Schematic of the interconnected network (left) and the interactions within each $p$-th subsystem (right).}
    \label{fig: PowerNetwork}
\end{figure}

The agents seek to stabilize their associated subsystems while subject to operational and structural limitations: 
The bus between any two subsystems $p,\tilde{p} \in \mathcal{P}$ is required to satisfy
\begin{equation} \label{eq: CapacityRestriction}
    -5 \leq \kappa^{p,\tilde{p}}\big(\theta^p(\tau) - \theta^{\tilde{p}}(\tau)\big) \leq 5, \quad \forall \tau \in \mathbb{R},
\end{equation}
to avoid overloading the corresponding channels.
Moreover, the actuation and communication infrastructure is such that both the deployment of controls and the transmission of information occur with a time delay of $0.1$ [units-of-time].

\subsection{Game formulation and the vGFNE-seeking problem}

This problem is formulated as a stochastic dynamic game 
\begin{equation*}
    \mathcal{G}_{\infty}^{\text{\text{LQ}}} \coloneqq (\mathcal{P}, \{ U^p \}_{p\in\mathcal{P}}, \{ J^p \}_{p\in\mathcal{P}}),
\end{equation*}
with players $\mathcal{P} = \{1, \ldots, 9\}$ and discrete-time dynamics
\begin{empheq}[left=\Sigma_{\mathcal{G}}:\empheqlbrace]{align*} 
    x_{t+1} &= A x_{t} + B_w w_t + \textstyle\sum_{p\in\mathcal{P}} B_u^p u^p_{t}, \quad x_0 = 0;\\
      y_{t} &= C x_{t} + D_w w_t.
\end{empheq}
defined by the block matrix $A = [A_{ij}]_{1\leq i,j, \leq N_P}$,
$$
    A_{ij} = \begin{cases}
        \begin{bmatrix} 1 & 0.1 \\ -\frac{0.1\sum_{j}k^{ij}}{m^i} & 1 - \frac{0.1d^i}{m^i} \end{bmatrix} & \text{ if } i = j  \\[1ex]
        \begin{bmatrix} 0 &  0  \\      \frac{0.1k^{ij}}{m^i}     &        0               \end{bmatrix} & \text{ otherwise}
    \end{cases}
$$
and $B_u^p = e_{p} \otimes [0 ~~ \frac{0.1}{m^p}]^{\tran}$ ($\forall p \in \mathcal{P}$), $C = I_{N_P} \otimes [1 ~~ 0]$, and $$(B_w,D_w) = \mathrm{blkdiag}\Big(I_{N_P} \otimes \begin{bmatrix} 0.01 & \\ & \frac{0.1}{m^p} \end{bmatrix}, 0.1I_{N_y}\Big).$$ 
This model is obtained by defining $x = (\theta^p, \dot{\theta}^p)_{p\in\mathcal{P}}$, then performing an Euler discretization to the state-space obtained from Eq.~\eqref{eq: SwingDynamics} using an interval of $\Delta \tau = 0.1$ units-of-time. 
The global measurement signal is defined as $y = (y^p)_{p\in\mathcal{P}}$. 
To ensure Assumption \ref{as: FeedbackNash_Equilbrium_Existence}(b), we also include a small artificial disturbance to the phase angle deviation $\theta^p$.
In our experiment, $\rho(A) = 1$ and thus the system is not stable. 
It is, however, both controllable and observable. 
The task of stabilizing each individual subsystem is encoded through the functional
$$
    J^p(u^p, u^{-p}) \coloneqq 
    \mathrm{E}\left[\sum_{t=0}^{\infty} \left\| \begin{bmatrix} W^p_x & W^p_u \end{bmatrix}\hspace*{-0.25em}\begin{bmatrix} x_t \\ u^p_t \end{bmatrix} \right\|_2^2 \right],
$$ 
with $[W_x^p ~~ W_u^p] = \mathrm{blkdiag}(e_p e_p^{\tran} \otimes 0.01I_{N_x}, ~I_{N_u^p})$. 
Finally, each $p$-th player's policy must be $(d_a,d_c)$-delayed with $d_a = d_c = 1$ and it must satisfy the operational constraints
\begin{equation} \label{eq: CapacityRestriction_U}
    U^{p}(u^{-p}) = \{ u^p \in \ell_{\infty}^{N_u}(\mathbb{N}) \mid G_x x_t \preceq 1, ~ t \in \mathbb{N} \},
\end{equation}
with matrix $G_x$ encoding the capacity limits from Eq.~\eqref{eq: CapacityRestriction}.

We simulate an execution of the game $\mathcal{G}_{\infty}^{\text{LQ}}$ in which the players are seeking an equilibrium policy by adhering to the vGFNE-seeking routine in Algorithm \ref{alg: FB_vGFNE-Seeking_SLS}. 
In this scenario, each player's policy is represented by a system level parametrization defined by FIR responses with $N = 16$ spectral components. 
The responses are constrained to have the sparsity patterns defined in Eq.~\eqref{eq: Structural_Constraints_SLP} and to satisfy a robust realization of the operational constraints with probability $\rho = 0.975$ (Eq.~\ref{eq: SLS_Chance_Constraint_SOC}). 
The policies are updated every $\Delta k = 32$ time steps (that is, $k = \lfloor t / \Delta k \rfloor$) with rate $\eta = \frac{M_{F_{\Phi}}}{L_{F_{\Phi}}^2} = \frac{0.0002}{0.0004} = 0.5$, computed and provided by the central coordinator to ensure convergence. 
We set $\Phi_{u|0} = \proj_{\mathcal{U}_{\Phi},\mathcal{G}}(0)$ as the initial profile.

% -----------------------------------------------------------------------------
\subsection{Simulation results and discussion}

Due to numerical limitations, we interrupt the updates at the $k_f$ first satisfying $\| \Phi_{u|k_f+1} - \Phi_{u|k_f} \|_{\ell_2} / \| \Phi_{u|k_f} \|_{\ell_2} \leq 10^{-15}$, when the policy updates become numerically negligible, and assume that $\Phi_{u|k_f} \approx \Phi_u^{\star} \in \mathbf{VI}(BR_{\Phi})$.
The convergence of the vGFNE-seeking algorithm to this fixed-point is shown in Figure \ref{fig: PowerNetwork_Convergence}.
The predicted convergence based on the rate $L_{(I-\eta F_{\Phi})} = \sqrt{1 - \eta(2 M_F - \eta L_F^2)} \approx 0.99995$ is also displayed.
The results show that, when adhering to Algorithm \ref{alg: FB_vGFNE-Seeking_SLS}, players converge to a vGFNE solution to $\mathcal{G}_{\infty}^{\text{LQ}}$ within $k_f = 10^4$ policy updates (or $t_f = 10^4 \times \Delta k$ time steps).
Interestingly, despite slow, the actual convergence of our vGFNE-seeking routine is still faster than that predicted by Corollary \ref{thm: vGFNE_FB_Convergence} by a factor of almost $100$.
We remark that this does not disqualify the usefulness of this corollary, as its primary purpose is to certificate that $\mathcal{G}_{\infty}^{\text{LQ}}$ has an unique solution to which Algorithm \ref{alg: FB_vGFNE-Seeking_SLS} converges.
In principle, once convergence is ensured, a potential acceleration can be achieved by incorporating inertia into Algorithm \ref{alg: FB_vGFNE-Seeking_SLS} \cite{Iutzeler2017}.
In this experiment, the seemingly slow convergence of the vGFNE-seeking routine is still not a practical issue:
If the power grid is operated at a timescale of seconds, it would take $(t_f \times \Delta \tau) = 32000$ seconds, or $8$ hours, for players to achieve an equilibrium strategy.
This is a relatively short period considering that such infrastructure is expected to be operational for years.
Finally, we again emphasize that Algorithm \ref{alg: FB_vGFNE-Seeking_SLS} allows the policy updates to occur alongside (and independently of) the actual operation of the system;
the players can continuously improve their policies without interrupting or destabilizing the grid. 

In Figure \ref{fig: PowerNetwork_Simulation}, we show the first hour of operation of the power grid while the policies are being updated (that is, for the first $t \leq 3600$ time steps or $k \leq \lfloor 3600/\Delta k \rfloor = 112$ policy updates).
We compare the closed-loop evolution with that obtained by an open-loop operation with $u^p_t = 0$ for all players $p \in \mathcal{P}$.
The results show that the policies $\hat{K}_k^p \coloneqq \hat{\Phi}_{uy|k}^p - \hat{\Phi}_{ux|k}^p \hat{\Phi}_{xx|k}^{-1} \hat{\Phi}_{xy|k}$ ($\forall p\in\mathcal{P}$) stabilize the network against the external disturbances, as observed through the noisy measurements of $\{ \theta^p \}_{p\in\mathcal{P}}$.
Moreover, the closed-loop system is not destabilized by the implementation of a new control policy every $\Delta k = 32$ time steps.
The policies are also shown to enforce the robust realization of the operational constraints: 
The buses $\kappa^{p,\tilde{p}}(\theta^p_t - \theta^{\tilde{p}}_t)$, between every $p,\tilde{p} \in \mathcal{P}$, are simultaneously within the safety limits for approximately 97\% of this simulation period.
Under open-loop operation, the safety limits are violated for the majority of the simulation period.
This performance is obtained at the expense of relatively large frequency loads $\{ u^p \}_{p\in\mathcal{P}}$.
Given the tuning $\{ W_x^p, W_u^p \}_{p\in\mathcal{P}}$, this is by design the optimal actions that the players must deploy while seeking an equilibrium policy. 

\begin{figure}[tb!] \centering
    \includegraphics[width=\columnwidth]{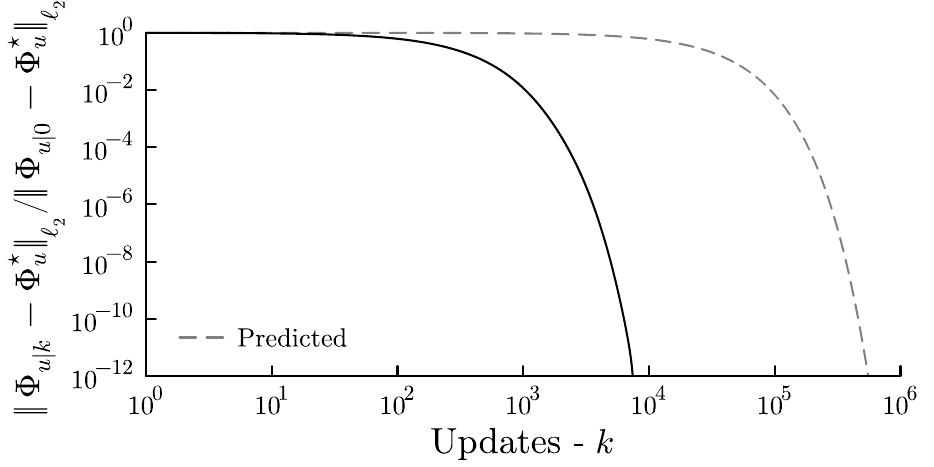}

    \caption{Power grid: Convergence of the vGFNE-seeking routine.}
    \label{fig: PowerNetwork_Convergence}
\end{figure}

\begin{figure}[tb!]
    \includegraphics[width=\columnwidth]{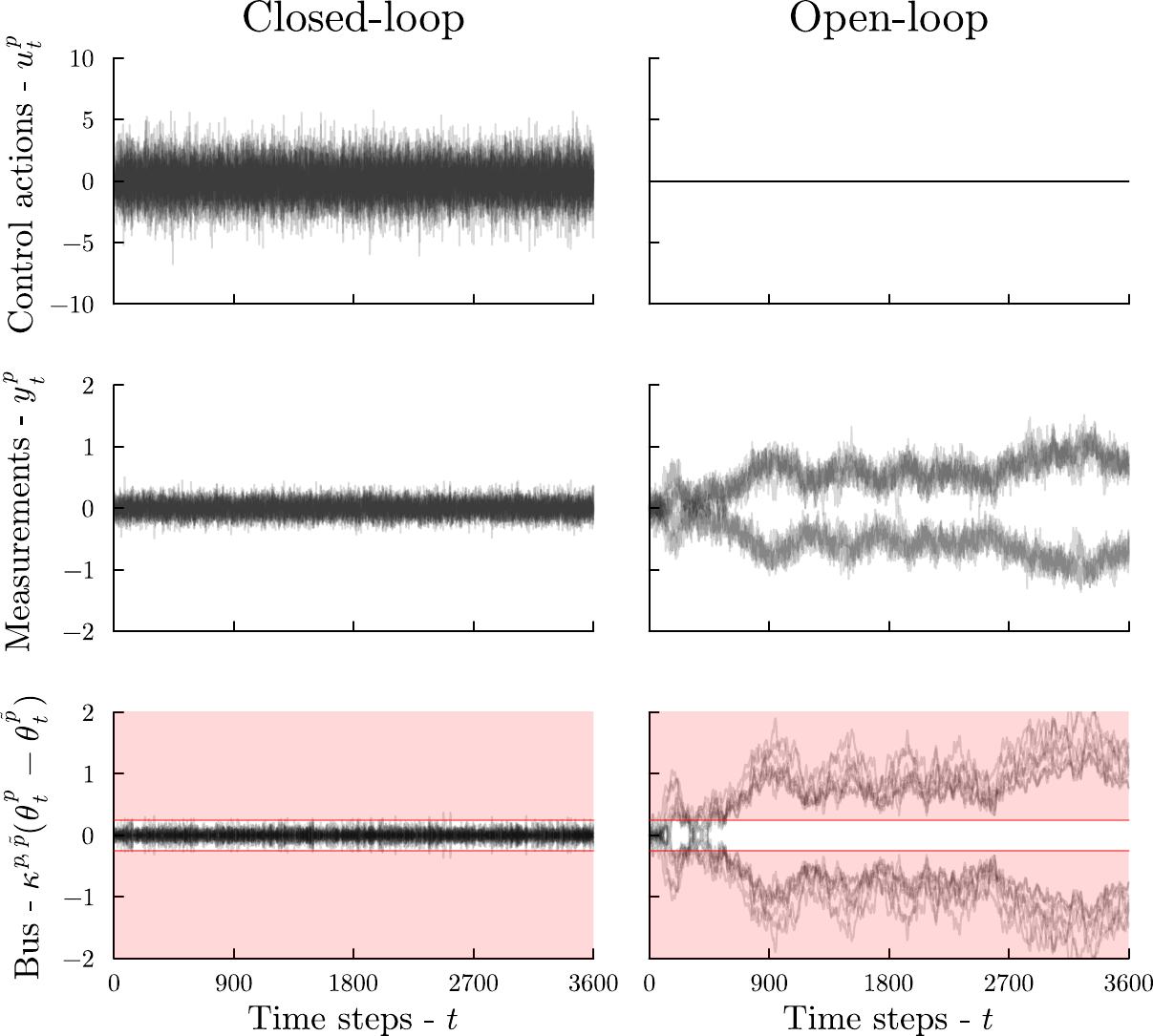}

    \caption{Power grid: Closed-loop (left panels) and open-loop (right panels) responses in terms of frequency loads $u^p$, phase angle measurements $y^p_t$, and pairwise buses $\kappa^{p,\tilde{p}}(\theta^p_t - \theta^{\tilde{p}}_t)$, for all $p,\tilde{p} \in \mathcal{P}$. The signals from each subsystem are plotted superimposed. The shaded region indicates the unsafe operation region of the network channels.}
    \label{fig: PowerNetwork_Simulation}
\end{figure}

Finally, the output-feedback policies $K_k^p$ ($\forall p\in\mathcal{P}$) are also shown to implement a communication structure based on the topology in Figure \ref{fig: PowerNetwork}.
To demonstrate this property, we simulate the closed-loop responses from the vGFNE policy, $K_{k_f} = (K^p_{k_f})_{p\in\mathcal{P}}$, when the system is initially at rest (i.e., $x_0 = 0$) and is then subjected to an impulse affecting the phase angle deviation of the central node (i.e., $B_w w_t = \delta_t e_9$).
The results (Figure \ref{fig: PowerNetwork_InfoStructure}) show that the flow of information follows the network topology depicted in Figure \ref{fig: PowerNetwork}:
At $t=1$, the impulse is perceived at the $5$-th subsystem and the observation $y_1^5$ is propagated to its neighbors $p \in \{ 4, 6, 8 \}$.
The players $p \in \{5, 4, 6, 8\}$ react to this disturbance after an action delay of $d_a = 1$ time steps.
At $t=3$, its effect is perceived at the subsystems $p \in \{ 4, 6, 8 \}$ and measurements $( y_1^{p}, y_2^{p}, y_3^p )_{p\in\{4,6,8\}}$ are transmitted to the two-hop neighbors $p \in \{1, 7, 9\}$.
This information pattern proceeds until all the players $p \in \mathcal{P}$ have been affected and are engaged in attenuating the disturbance.
The magnitude of the applied control actions is shown to be higher the closer the subsystems are to the central node $p=5$.
Additionally, the policies are observed to implement \textit{deadbeat} control: 
The state is stirred to zero instantaneously after $N = 16$ control actions.
This behavior (known as \textit{time-localization} in the SLS literature \cite{Anderson2019}) explains the aforementioned large control actions, as deadbeat controllers are notoriously aggressive for small $N$ \cite{Astrom1997}.
Interestingly, this property also implies that disturbances can be attenuated before affecting subsystems that are more than $N$-hops distant from where they enter the network.

\begin{figure}[t!]
    \includegraphics[width=\columnwidth]{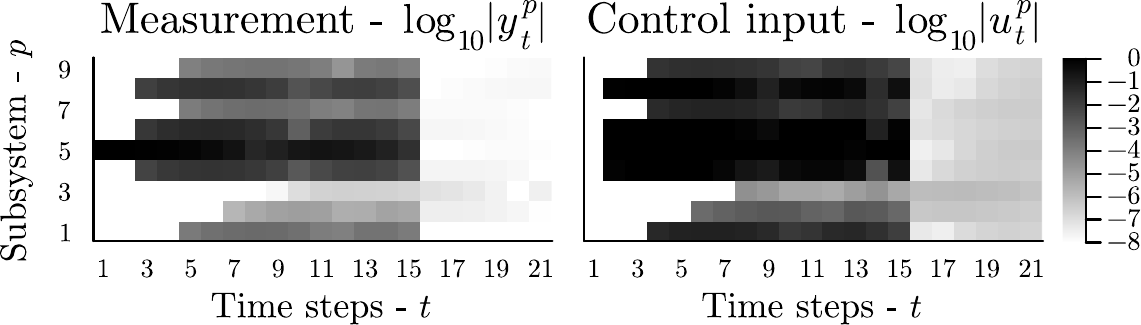}

    \caption{Power grid: Closed-loop impulse response in terms of measurements $y^p$ and control inputs $u^p$, for each individual subsystem $p \in \mathcal{P}$.}
    \label{fig: PowerNetwork_InfoStructure}
\end{figure}

% -----------------------------------------------------------------------------
\section{Concluding remarks} \label{sec: Concluding_Remarks}

This work presented a GFNE-seeking algorithm for noncooperative games with dynamics that are linear, stochastic, potentially unstable, and partially observed.
We first considered the equivalent representation of each player's policies as their corresponding closed-loop responses to the noise (given others' policies), then we designed a vGFNE-seeking routine based on well-established operator splitting algorithms.  
Under this system level parametrization, the GFNE-seeking algorithm is applicable to problems in which the policies are required to be stabilizing and to satisfy operational (on the state and input signals) and structural (on the policy itself) constraints.
Using results from operator theory, we derived conditions for the existence and uniqueness of a vGFNE solution and established convergence certificates for our algorithm.
After its main aspects are presented, the algorithm is demonstrated on the task of stabilizing a partially observed and decentralized power grid.
The results demonstrate its efficacy: When adhering to this routine, players approach a GFNE solution while still complying with operational and communication constraints.

\appendix[Proof of Theorem \ref{thm: FPhi_Lipschitz_Monotone}]

From Eq.~\eqref{eq: Pseudogradient_Affine_Players}, the pseudo-gradient has the compact form
$$
    F_{\Phi}(\Phi_u) = 2(D_{\Phi}^{\tran}H_{\Phi}) \Phi_u (W_w W_w^{\tran}) + h 
$$
given matrices $D_{\Phi} = \mathrm{blkdiag}(H_{\Phi}^{pp})_{p\in\mathcal{P}}$, $H_{\Phi} = [H_{\Phi}^{p\tilde{p}}]_{p,\tilde{p}\in\mathcal{P}}$, and $h = (h^{p})_{p\in\mathcal{P}}$.
In vectorized form, this operator becomes $\mathrm{vec}(F_{\Phi}(\Phi_u)) = 2\big((W_w W_w^{\tran})\otimes(D_{\Phi}^{\tran} H_{\Phi})\big)\mathrm{vec}(\Phi_u) + \mathrm{vec}(g)$. 
Since the strong-monotonicity constant of an affine operator $F(x) = Ax+b$ is $(1/2)\lambda_{\min}(A+A^{\tran})$ \cite{Ryu2022}, we have that
\begin{align*}
    M_{F_{\Phi}} 
        &= \lambda_{\min}\big( (W_w W_w^{\tran}) \otimes D_{\Phi}^{\tran} H_{\Phi} +  ((W_w W_w^{\tran}) \otimes D_{\Phi}^{\tran} H_{\Phi})^{\tran} \big) \\ 
        &= \lambda_{\min}\big( (W_w W_w^{\tran}) \otimes (D_{\Phi}^{\tran} H_{\Phi} + H_{\Phi}^{\tran} D_{\Phi} ) \big) \\
        &= \lambda_{\min}\big( D_{\Phi}^{\tran} H_{\Phi} + H_{\Phi}^{\tran} D_{\Phi} \big) \sigma_{\min}^2(W_w)
\end{align*}
where we used the distributivity of the Kronecker product, the fact that $\lambda_{\min}(A \otimes B) = \lambda_{\min}(B)\lambda_{\min}(A)$, and the property $\lambda_{\min}(A A^{\tran}) = \sigma_{\min}^2(A)$, for any matrices $A$ and $B$ \cite{Horn1991}.
The tightest Lipschitz constant of $F_{\Phi}$ is obtained from the spectral norm $\|2(W_w W_w^{\tran}) \otimes (D_{\Phi}^{\tran} H_{\Phi})\|_2$ \cite{Ryu2022}, or, equivalently,
\begin{align*}
    L_{F_{\Phi}}
        &= 2\sigma_{\max}\big( (W_w W_w^{\tran}) \otimes (D_{\Phi}^{\tran} H_{\Phi}) \big) \\
        &= 2\sigma_{\max}\big( D_{\Phi}^{\tran} H_{\Phi} \big)\sigma_{\max}^2\big( W_w \big),
\end{align*}%
using the fact that $\lambda_{\max}(A \otimes B) = \lambda_{\max}(B)\lambda_{\max}(A)$ and that $\lambda_{\max}(A A^{\tran}) = \sigma_{\max}^2(A)$ for any matrices $A$ and $B$.  

% -----------------------------------------------------------------------------
\bibliographystyle{ieeetr}

% -----------------------------------------------------------------------------
\end{document}